\documentclass[11pt]{amsart}

\usepackage[utf8]{inputenc}
\usepackage[american]{babel}
\usepackage{amsmath,amssymb,amsthm}
\usepackage{newpxtext,newpxmath}
\usepackage{enumerate}
\usepackage{csquotes}
\usepackage{calrsfs}
\usepackage{dsfont}
\usepackage[backend=biber,style=alphabetic,url=false,doi=false,isbn=false]{biblatex}
\usepackage[left=3.5cm,right=3.5cm,top=3cm,bottom=3cm]{geometry}
\usepackage{color}
\usepackage[hidelinks]{hyperref}
\usepackage{cleveref}

\newcommand{\1}{\mathds 1}
\newcommand{\IC}{\mathbb C}
\newcommand{\IL}{\mathbb L}
\newcommand{\IN}{\mathbb N}
\newcommand{\IR}{\mathbb R}

 \newcommand{\E}{\mathcal E}
 \renewcommand{\H}{\mathcal H}
 \newcommand{\J}{\mathcal J}
\renewcommand{\L}{\mathcal L}
\newcommand{\U}{\mathcal U}
\newcommand{\V}{\mathcal V}
\newcommand{\W}{\mathcal W}

\renewcommand{\AA}{\mathfrak A}

 \newcommand{\bin}{\mathrm{bin}}
\newcommand{\id}{\mathrm{id}}
\newcommand{\op}{\mathrm{op}}
\newcommand{\tr}{\mathrm{tr}}
\newcommand{\spec}{\mathrm{sp}}

\newcommand{\dom}{\operatorname{dom}}
\renewcommand{\Im}{\operatorname{Im}}
\renewcommand{\Re}{\operatorname{Re}}

\renewcommand{\epsilon}{\varepsilon}
\renewcommand{\phi}{\varphi}

\newcommand{\abs}[1]{\lvert#1\rvert}
\newcommand{\norm}[1]{\lVert#1\rVert}
\newcommand{\ip}[1]{\langle#1\rangle}

\theoremstyle{plain}
\newtheorem{proposition}{Proposition}[section]
\newtheorem{lemma}[proposition]{Lemma}
\newtheorem{theorem}[proposition]{Theorem}

\newtheorem*{theorem*}{Theorem}

\theoremstyle{definition}
\newtheorem{definition}[proposition]{Definition}

\theoremstyle{remark}
\newtheorem{remark}[proposition]{Remark}

\title{Derivations and KMS-Symmetric Quantum Markov Semigroups}

\author{Matthijs Vernooij}
\address{Delft University of Technology, Faculty EEMCS/DIAM, P.O.Box 5031, 2600 GA Delft, The Netherlands}
\email{m.n.a.vernooij@tudelft.nl}

\author{Melchior Wirth}
\address{Institute of Science and Technology Austria (ISTA), Am Campus 1, 3400 Klosterneuburg, Austria}
\email{melchior.wirth@ist.ac.at}

\begin{document}

\begin{abstract}
We prove that the generator of the $L^2$ implementation of a KMS-symmetric quantum Markov semigroup can be expressed as the square of a derivation with values in a Hilbert bimodule, extending earlier results by Cipriani and Sauvageot for tracially symmetric semigroups and the second-named author for GNS-symmetric semigroups. This result hinges on the introduction of a new completely positive map on the algebra of bounded operators on the GNS Hilbert space. This transformation maps symmetric Markov operators to symmetric Markov operators and is essential to obtain the required inner product on the Hilbert bimodule.
\end{abstract}

\maketitle

%\tableofcontents

\section{Introduction}

Quantum Markov semigroups are a versatile tool that has found applications not only in quantum statistical mechanics, where they were originally introduced in the description of certain open quantum systems \cite{Ali76,GKS76,Kos72,Lin76}, but also in various purely mathematical fields such as noncommutative harmonic analysis \cite{JX07,JMP14}, noncommutative probability \cite{Bia03,CFK14}, noncommutative geometry \cite{Arh23,CS03a,Sau96} and the structure theory of von Neumann algebras \cite{CS15,CS17,Pet09a}.

One central question from the beginning was to describe the generators of quantum Markov semigroups. For quantum Markov semigroups acting on matrix algebras, a characterization of their generators was given by Lindblad \cite{Lin76} and Gorini--Kossakowski--Sudarshan \cite{GKS76} and later extended to generators of uniformly continuous quantum Markov semigroups on arbitrary von Neumann algebras by Christensen--Evans \cite{CE79}. While partial results are known in particular for type I factors \cite{AZ15,Dav79,Hol95}, a similarly explicit description of unbounded generators of quantum Markov semigroups on arbitrary von Neumann algebras seems out of reach.

Both for the modelling of open quantum systems and purely mathematical questions in noncommutative probability, operator algebra theory, etc., one is often not interested in arbitrary quantum Markov semigroups, but quantum Markov semigroups that are symmetric with respect to a reference state or weight. In quantum statistical mechanics, these describe open systems coupled to a heat bath in thermal equilibrium. From the mathematical standpoint, symmetry with respect to a reference state allows to extend the semigroups to symmetric semigroups on the GNS Hilbert space, which makes the powerful tools for self-adjoint Hilbert space operators available.

If the reference state or weight is a trace, there is an unambiguous notion of symmetry, called tracial symmetry. The study of tracially symmetric quantum Markov semigroups through their associated quadratic forms, so-called Dirichlet forms, was initiated by Albeverio and H{\o}egh-Krohn \cite{AH77} and further developed by Lindsay and Davies \cite{DL92,DL93}. The (Hilbert space) generators of tracially symmetric quantum Markov semigroups have been characterized by Cipriani and Sauvageot \cite{CS03b} to be of the form $\delta^\ast\delta$, where $\delta$ is a derivation with values in a Hilbert bimodule. This result has lead to a far range of applications from analysis
on fractals \cite{HT13} and metric graphs \cite{BK19} over noncommutative geometry \cite{CS03b},
noncommutative probability \cite{Dab10, JZ15}, quantum optimal transport \cite{Wir20, WZ21} to the structure theory of von Neumann algebras and in particular Popa’s deformation
and rigidity theory \cite{Pet09b, DI16, Cas21, CIW21}.

Despite this success, the notion of tracial symmetry is somewhat limiting. For one, von Neumann algebras with a type III summand do not even admit a (faithful normal) trace. But even on semi-finite von Neumann algebras, plenty of quantum Markov semigroups of interest are not tracially symmetric. For example, in many models of open quantum systems one considers quantum Markov semigroups symmetric with respect to a Gibbs state, which is only a trace in the infinite temperature limit.

In the case when the reference state is not a trace, there are several non-equivalent notions of symmetry, for example GNS symmetry and KMS symmetry. If the state is of the form $\tr(\cdot\,\rho)$ on a type I factor, then GNS symmetry is symmetry with respect to the inner product $(x,y)\mapsto\tr(x^\ast y\rho)$, while KMS symmetry is symmetry with respect to the inner product $(x,y)\mapsto \tr(x^\ast\rho^{1/2}y \rho^{1/2})$.

GNS symmetry is the strongest one of the symmetry conditions usually considered, and it also implies commutation with the modular group, which makes the structure of GNS-symmetric quantum Markov semigroups particularly nice and rich. For GNS-symmetric quantum Markov semigroups on matrix algebras, Alicki's theorem \cite{Ali76} gives a characterization of their generators in the spirit of the results of Lindblad and Gorini--Kossakowski--Sudarshan. But it can also be recast as a representation of the generator as square of a derivation, thus presenting an analogue of the result of Cipriani and Sauvageot mentioned above (see \cite{CM17}). In this process one loses the property that the left and right action are $\ast$-homomorphisms, but this is unavoidable, as was shown by the first-named author \cite{Ver22}.

This result of Alicki has played a central role in recent research motivated by quantum information theory, in particular for the development of a dynamical quantum optimal transport distance \cite{CM17}, relating hypercontractivity and logarithmic Sobolev inequalities for quantum systems \cite{Bar17} and the proof of the complete modified logarithmic Sobolev inequality for finite-dimensional GNS-symmetric quantum Markov semigroups \cite{GR22}.

Moving beyond matrix algebras, the second-named author established a version of the Christensen--Evans theorem for generators of uniformly continuous GNS-symmetric quantum Markov semigroups \cite{Wir22a} and a generalization of the result of Cipriani and Sauvageot for generators of arbitrary GNS-symmetric quantum Markov semigroups \cite{Wir22b}.

Let us describe the latter result in some more detail. If a quantum Markov semigroup $(\Phi_t)$ is GNS-symmetric with respect to the state (or more generally weight) $\phi$, then it has a GNS implementation as strongly continuous semigroup $(T_t)$ on the GNS Hilbert space $L^2(M,\phi)$. Let $\E$ denote the associated quadratic form and $\AA_\phi$ the maximal Tomita algebra induced by $\phi$. It is shown in \cite{Wir22b} that
\begin{equation*}
    \AA_\E=\{a\in \AA_\phi\mid \Delta_\phi^z(a)\in \dom(\E)\text{ for all }z\in\IC\}
\end{equation*}
is a Tomita algebra and a form core for $\E$. Moreover, there exists a Hilbert space $\H$ with a left action of $\AA_\E$ that is a $\ast$-homomorphism for $^\sharp$ and a right action that is a $\ast$-homomorphism for $^\flat$ and a closable operator $\delta\colon\AA_\E\to \H$ satisfying the Leibniz rule $\delta(ab)=a\delta(b)+\delta(a)b$ such that
\begin{equation*}
    \E(a,b)=\langle\delta(a),\delta(b)\rangle_\H
\end{equation*}
for $a,b\in\AA_\E$. Moreover, $\H$ carries an anti-unitary involution and a strongly continuous unitary group with certain compatibility conditions that reflect the commutation of $(T_t)$ with the modular operator and modular conjugation.

It is then natural to ask whether this relation between $L^2$ generators of quantum Markov semigroups and derivations can be extended to KMS-symmetric semigroups as KMS symmetry can be seen as the more natural assumption in some contexts. For one, every completely positive map can be decomposed as a linear combination of KMS-symmetric ones using the Accardi--Cecchini adjoint \cite{AC82}, while for GNS symmetry the commutation with the modular group poses an algebraic constraint. This makes KMS-symmetric quantum Markov semigroups more suitable for various applications, such as the characterization of the Haagerup property in terms of KMS-symmetric quantum Markov semigroups \cite{CS15}, while the same property for GNS-symmetric semigroups is more restrictive. But also in quantum statistical mechanics, irreversible open quantum systems are often modeled by quantum Markov semigroups that are only KMS-symmetric rather than GNS-symmetric,  such as the heat-bath dynamics introduced in \cite{KB16} for example.

The lack of commutation with the modular group poses a serious challenge. For example, many questions regarding noncommutative $L^p$ spaces can be reduced to $L^p$ spaces with respect to a trace by Haagerup's reduction method \cite{HJX10}, but commutation with the modular group is necessary for maps to be compatible with this reduction procedure.

Even for KMS-symmetric quantum Markov semigroups on type I factors, when explicit representations of the generator are known \cite{FU10,AC21}, it is not obvious if the generator can be expressed in the form $\delta^\ast\delta$ for a derivation.

For these reasons it has not been clear if one should even expect a Cipriani--Sauvageot-type result for KMS-symmetric quantum Markov semigroups. In this article we show that this is indeed the case, not only on matrix algebras, but arbitrary von Neumann algebras. For generators of uniformly continuous quantum Markov semigroups, our main result is the following (Theorems \ref{thm:existence:deriv_bounded}, \ref{thm:uniqueness} in the main part). Here $\L_2$ denotes the KMS implementation of $\L$ on $L^2(M)$.

\begin{theorem*}
Let $(\Phi_t)$ be a uniformly continuous KMS-symmetric quantum Markov semigroup on $M$ and let $\L$ denote its generator.  There exists a correspondence $\H$, an anti-linear involution $\J\colon \H\to \H$ and a bounded operator $\delta\colon L^2(M)\to \H$ satisfying
\begin{enumerate}[\indent (a)]
    \item $\J(x\xi y)=y^\ast(\J\xi)x^\ast$ for all $x,y\in M$ and $\xi\in \H$,
    \item $\delta(Ja)=\J\delta(a)$ for all $a\in L^2(M)$,
    \item $\delta(ab)=\pi_l(a)\cdot\delta(b)+\delta(a)\cdot J\pi_r(b)^\ast J$ for all $a\in M\phi^{1/2}$, $b\in \phi^{1/2}M$,
    \item $\overline{\mathrm{lin}}\{\delta(a)x\mid a\in L^2(M),\,x\in M\}=\H$
\end{enumerate}
such that
\begin{equation*}
    \L_2=\delta^\ast \delta.
\end{equation*}
Moreover, there exists $\xi\in \H$ such that 
\begin{equation*}
    \delta(a)=\pi_l(a)\xi-\xi(J\pi_r(a)^\ast J)
\end{equation*}
for $a\in M\phi^{1/2}\cap \phi^{1/2}M$.

Furthermore, a triple $(\H,\J,\delta)$ satisfying (a)--(d) is uniquely determined by $\delta^\ast\delta$ up to isomorphism.
\end{theorem*}

For KMS-symmetric quantum Markov semigroups that are not  uniformly continuous we do not have a uniqueness result, but we can still prove existence in the following form (Theorems \ref{thm:algebra_Dirichlet_form}, \ref{thm:existence_deriv_unbounded} in the main part).

\begin{theorem*}
Let $(\Phi_t)$ be a KMS-symmetric quantum Markov semigroup on $M$ with generator $\L$. If $a\in \dom(\L_2^{1/2})\cap M\phi^{1/2}$ and $b\in \dom(\L_2^{1/2})\cap \phi^{1/2}M$, then $ab\in \dom(\L_2^{1/2})$.

Moreover, there exists a Hilbert space $\H$ with commuting left and right actions of $M$, an anti-unitary involution $\J\colon \H\to\H$ such that
\begin{equation*}
    \J(x\xi y)=y^\ast(\J\xi)x^\ast
\end{equation*}
for $x,y\in M$ and $\xi\in \H$, a closed operator $\delta\colon \dom(\L_2^{1/2})\to \H$ such that $\J\delta=\delta J$ and
\begin{equation*}
    \delta(ab)=\pi_l(a)\cdot\delta(b)+\delta(a)\cdot J\pi_r(a)^\ast J
\end{equation*}
for $a\in \dom(\L_2^{1/2})\cap M\phi^{1/2}$, $b\in \dom(\L_2^{1/2})\cap \phi^{1/2}M$, and
\begin{equation*}
    \L_2=\delta^\ast\delta.
\end{equation*}
\end{theorem*}

To establish these results, it requires a fundamentally new tool in the form of a quantum channel on $B(L^2(M))$, which we call the \emph{$\V$-transform}. Formally, the $\V$-transform of $T\in B(L^2(M))$ is the solution $S$ of the equation
\begin{equation*}
    \frac 1 2(\Delta^{1/4}S\Delta^{-1/4}+\Delta^{-1/4}S\Delta^{1/4})=T,
\end{equation*}
where $\Delta$ is the modular operator.

A remarkable fact about this map is that it maps positivity-preserving maps (with respect to the self-dual cone induced by $\phi$) to positivity-preserving maps. This property is key for the existence of the Hilbert space $\H$ in our main results.

Let us briefly summarize the outline of this article. To make the proof strategy transparent without the technical difficulties occuring for general von Neumann algebras, but also to make the main results more accessible for researchers in the quantum information theory community, we first develop the $\V$-transform and prove our main result for matrix algebras in \Cref{sec:matrix_algebras}. In \Cref{sec:V-transform} we define the $\V$-transform on general von Neumann algebras and establish some of its properties in particular regarding positvity preservation. In \Cref{sec:uniformly_bounded} we prove our main results on existence and uniqueness of derivations associated with uniformly continuous KMS-symmetric quantum Markov semigroups. Finally, in \Cref{sec:unbounded} we show the existence of derivations associated with not necessarily uniformly continuous semigroups.

\subsection*{Acknowledgments} The authors are grateful to Martijn Caspers for helpful comments on a preliminary version of this manuscript. M. V. was supported by the NWO Vidi grant VI.Vidi.192.018 `Non-commutative harmonic analysis and rigidity of operator
algebras'. M. W. was funded by the Austrian Science Fund (FWF) under the Esprit Programme [ESP 156]. For the purpose of Open Access, the authors have applied a CC
BY public copyright licence to any Author Accepted Manuscript (AAM) version arising from this submission.

\section{Quantum Markov Semigroups and Derivations on Matrix Algebras}\label{sec:matrix_algebras}

In this section we demonstrate the connection between KMS-symmetric quantum Markov semigroups and derivations in the case of matrix algebras. We first prove a finite-dimensional version of the main result of this article, which allows to express generators of KMS-symmetric quantum Markov semigroups as squares of derivations in a suitable sense (\Cref{thm:existence_derivation_matrix}). We then use the simple structure of bimodules over matrix algebras to give a more explicit expression for the quadratic form associated with the generator of a KMS-symmetric quantum Markov semigroups (\Cref{thm:commutator_sqrt_matrix}). As in the general case treated in the next sections, the crucial technical tool is the $\V$-transform, which will be introduced in Subsection \ref{subsec:V-transform_matrix}.

Let us start with some notation. We write $M_n(\IC)$ for the algebra of $n\times n$ matrices over the complex numbers, $I_n$ for the identity matrix in $M_n(\IC)$, and $\id_n$ for the identity map from $M_n(\IC)$ to itself. The norm $\norm\cdot$ always denotes the operator norm, either for elements of $M_n(\IC)$ or for linear maps from $M_n(\IC)$ to itself.

A linear map $\Phi\colon M_n(\IC)\to M_n(\IC)$ is called
\begin{itemize}
    \item \emph{completely positive} if, for all $A_1,\dots,A_n,B_1,\dots,B_n\in M_n(\IC)$,
    \begin{equation*}
        \sum_{j,k=1}^n B_j^\ast\Phi(A_j^\ast A_k)B_k\geq 0;
    \end{equation*}
    \item \emph{conditionally completely negative} if, for all $A_1,\dots,A_n,B_1,\dots,B_n\in M_n(\IC)$ with $\sum_{j=1}^nA_jB_j=0$,
    \begin{equation*}
        \sum_{j,k=1}^n B_j^\ast\Phi(A_j^\ast A_k)B_k\leq 0;
    \end{equation*}
    \item \emph{unital} if $\Phi(I_n)=I_n$.
\end{itemize}

% A linear map $\Phi\colon M_n(\IC)\to M_n(\IC)$ is called \emph{completely positive} if
% \begin{equation*}
%     \sum_{j,k=1}^n B_j^\ast\Phi(A_j^\ast A_k)B_k\geq 0    
% \end{equation*}
% for all $A_1,\dots,A_n,B_1,\dots,B_n\in M_n(\IC)$. It is called \emph{unital} if $\Phi(I_n)=I_n$.

A \emph{quantum Markov semigroup} is a family $(\Phi_t)_{t\geq 0}$ of unital completely positive maps on $M_n(\IC)$ such that
\begin{itemize}
    \item $\Phi_0=\id_n$,
    \item $\Phi_s\Phi_t=\Phi_{s+t}$ for all $s,t\geq 0$,
    \item $\lim_{t\to 0}\lVert \Phi_t-\id_n\rVert=0$.
\end{itemize}
If $(\Phi_t)$ is a quantum Markov semigroup on $M_n(\IC)$, then the limit
\begin{equation*}
    \L=\lim_{t\to 0}\frac 1 t(\id_n-\Phi_t)
\end{equation*}
exists and is called the \emph{generator} of $(\Phi_t)$. It is the unique linear operator on $M_n(\IC)$ such that $e^{-t\L}=\Phi_t$ for all $t\geq 0$. %\Red{From the definition we see that the generator of a quantum Markov semigroup is conditionally completely negative. Conversely, any conditionally completely negative map $\L$ satisfying $\L(I_n)=0$ and $\L(A^\ast)=\L(A)^\ast$ generates a quantum Markov semigroup \cite{??}.}

Now fix a density matrix $\rho\in M_n(\IC)$, that is, a positive matrix with trace $1$, and assume that $\rho$ is invertible. The KMS inner product induced by $\rho$ is defined as
\begin{equation*}
    \langle \cdot,\cdot\rangle_\rho\colon M_n(\IC)\times M_n(\IC)\to \IC,\,(A,B)\mapsto \tr(A^\ast \rho^{1/2}B\rho^{1/2}).
\end{equation*}
If $\Phi\colon M_n(\IC)\to M_n(\IC)$ is a linear map, we write $\Phi^{\dagger}$ for its adjoint with respect to $\langle\cdot,\cdot\rangle_\rho$ and we say that $\Phi$ is \emph{KMS-symmetric} if $\Phi^\dagger=\Phi$.

A quantum Markov semigroup $(\Phi_t)$ is called \emph{KMS-symmetric} if $\Phi_t$ is KMS-symmetric for all $t\geq 0$. Equivalently, $(\Phi_t)$ is KMS-symmetric if and only if its generator is KMS-symmetric.

The \emph{modular group} (or rather its analytic continuation) of $\rho$ is the family $(\sigma_z)_{z\in\IC}$ of algebra homomorphisms on $M_n(\IC)$ defined by
\begin{equation*}
    \sigma_z(A)=\rho^{iz}A\rho^{-iz}
\end{equation*}
for $A\in M_n(\IC)$ and $z\in\IC$. Note that $\sigma_z^\dagger=\sigma_{-\bar z}$.

If a KMS-symmetric operator on $M_n(\IC)$ commutes with the modular group, then it is called GNS-symmetric (this is equivalent to the usual definition of GNS symmetry by \cite[Lemma 2.5]{CM17}).

According to \cite[Theorem 4.4]{AC21}, the generator $\L$ of a KMS-symmetric quantum Markov semigroup is of the form
\begin{equation*}
    \L(A)=(1+\sigma_{-i/2})^{-1}(\Psi(I_n))A+A(1+\sigma_{i/2})^{-1}(\Psi(I_n))-\Psi(A)
\end{equation*}
for some KMS-symmetric completely positive map $\Psi\colon M_n(\IC)\to M_n(\IC)$.

The main goal of this section is to show that the sesquilinear form associated with $\L$ can be written as
\begin{equation*}
    \langle \L(A),B\rangle_\rho=\sum_{j=1}^N \langle [V_j,A],[V_j,B]\rangle_\rho
\end{equation*}
with matrices $V_1,\dots,V_N\in M_n(\IC)$.

\subsection{\texorpdfstring{The $\V$-Transform}{The V-Transform}}\label{subsec:V-transform_matrix}

We write $B(M_n(\IC))$ for the space of all linear maps from $M_n(\IC)$ to itself. This space is generated by left and right multiplication operators in the following sense. For $A\in M_n(\IC)$ let
\begin{equation*}
    \IL_A,\IR_A\colon M_n(\IC)\to M_n(\IC),\,\IL_A(X)=AX,\,\IR_A(X)=XA.
\end{equation*}
By \cite[Lemma A.1]{CM17}, the linear span of $\{\IL_A\IR_B\mid A,B\in M_n(\IC)\}$ is $B(M_n(\IC))$.

The $\V$-transform is a linear map on $B(M_n(\IC))$, which is most conveniently defined through its inverse. Let
\begin{equation*}
    \W\colon B(M_n(\IC))\to B(M_n(\IC)),\,\Phi\mapsto \frac 1 2(\sigma_{i/4}\Phi\sigma_{-i/4}+\sigma_{-i/4}\Phi\sigma_{i/4}).
\end{equation*}
In particular, if $\Phi=\IL_A\IR_B$, then
\begin{equation*}
    \W(\Phi)=\frac 1 2(\IL_{\sigma_{i/4}(A)}\IR_{\sigma_{i/4}(B)}+\IL_{\sigma_{-i/4}(A)}\IR_{\sigma_{-i/4}(B)}).
\end{equation*}

\begin{proposition}\label{prop:V-trafo_inverse_matrix}
The map $\W$ is invertible with inverse given by
\begin{equation*}
    \W^{-1}(\Phi)=2\int_0^\infty \sigma_{-i/4}e^{-r\sigma_{-i/2}} \Phi \sigma_{-i/4}e^{-r\sigma_{-i/2}}\,dr
\end{equation*}
for $\Phi\in B(M_n(\IC))$.

In particular, if $A,B\in M_n(\IC)$, then
\begin{equation*}
    \W^{-1}(\IL_A\IR_B)=2\int_0^\infty \IL_{\sigma_{i/4}(e^{-r\sigma_{i/2}}(A))}\IR_{\sigma_{-i/4}(e^{-r\sigma_{-i/2}}(B))}\,dr.
\end{equation*}
\end{proposition}
\begin{proof}
Since $\sigma_{-i/2}$ is an invertible operator on $M_n(\IC)$ and $\tr(\sigma_{-i/2}(A)^\ast A)\geq 0$ for all $A\in M_n(\IC)$, the spectrum of $\sigma_{-i/2}$ consists of strictly positive numbers. Let $\lambda$ denote the smallest eigenvalue of $\sigma_{-i/2}$. By the spectral theorem,
\begin{equation*}
    \norm{e^{-r\sigma_{-i/2}}}\leq e^{-\lambda r}
\end{equation*}
for all $r\geq 0$. It follows that for $\Phi\in B(M_n(\IC))$ we have
\begin{equation*}
    \norm{\sigma_{-i/4}e^{-r\sigma_{-i/2}} \Phi \sigma_{-i/4}e^{-r\sigma_{-i/2}}}\leq e^{-2\lambda r}\norm{\sigma_{-i/4}}^2 \norm{\Phi}.
\end{equation*}
Therefore the integral
\begin{equation*}
    \int_0^\infty \sigma_{-i/4}e^{-r\sigma_{-i/2}} \Phi \sigma_{-i/4}e^{-r\sigma_{-i/2}}\,dr
\end{equation*}
converges absolutely.

Moreover,
\begin{align*}
    &\W\left(2\int_0^\infty \sigma_{-i/4}e^{-r\sigma_{-i/2}} \Phi \sigma_{-i/4}e^{-r\sigma_{-i/2}}\,dr\right)\\
    &\qquad=\int_0^\infty e^{-r\sigma_{-i/2}}(\sigma_{-i/2}\Phi+\Phi\sigma_{-i/2}) e^{-r\sigma_{-i/2}}\,dr\\
    &\qquad=-\int_0^\infty \frac{d}{dr}(e^{-r\sigma_{-i/2}}\Phi e^{-r\sigma_{-i/2}})\,dr\\
    &\qquad=\Phi.
\end{align*}
Thus $\W$ is invertible and the claimed integral expression for the inverse holds. Similar arguments yield the integral formula for $\W^{-1}(\IL_A\IR_B)$.
\end{proof}

\begin{definition}
We call the inverse of $\W$ the $\V$-transform and denote it by $\V$. For $\Phi\in B(M_n(\IC))$ we also write $\check \Phi$ for $\V(\Phi)$.
\end{definition}

\begin{lemma}\label{lem:properties_V-trafo_matrix}
The $\V$-transform is a bijective linear map on $B(M_n(\IC))$ with the following properties.
\begin{enumerate}[(i)]
    \item If $\Phi\in B(M_n(\IC))$ is KMS-symmetric, then $\V(\Phi)$ is KMS-symmetric.
    \item If $\Phi\in B(M_n(\IC))$ is completely positive, then $\V(\Phi)$ is completely positive.
    \item $\V(\id_n)=\id_n$.
    %\item \Red{If $\Phi\in B(M_n(\IC))$ is the generator of a quantum Markov semigroup, then $\V(\Phi)$ is the generator of a QMS.}
\end{enumerate}
\end{lemma}
\begin{proof}
\begin{enumerate}[(i)]
    \item Let $\Phi\in B(M_n(\IC))$. By \Cref{prop:V-trafo_inverse_matrix} we have
    \begin{equation*}
        \V(\Phi)=2\int_0^\infty \sigma_{-i/4}e^{-r\sigma_{-i/2}} \Phi \sigma_{-i/4}e^{-r\sigma_{-i/2}}\,dr.
    \end{equation*}
    Since $\sigma_{-i/4}$ is KMS-symmetric, so is $e^{-r\sigma_{-i/2}}$. Thus the KMS adjoint of $\V(\Phi)$ satisfies
    \begin{align*}
        \V(\Phi)^\dagger=2\int_0^\infty \sigma_{-i/4}e^{-r\sigma_{-i/2}} \Phi^\dagger \sigma_{-i/4}e^{-r\sigma_{-i/2}}\,dr.
    \end{align*}
    In particular, if $\Phi$ is KMS-symmetric, so is $\V(\Phi)$.
    \item If $\Phi$ is completely positive, by Kraus' theorem there exist $V_1,\dots,V_N\in M_n(\IC)$ such that 
    \begin{equation*}
        \Phi=\sum_{j=1}^N \IL_{V_j^\ast}\IR_{V_j}.
    \end{equation*}
    From \Cref{prop:V-trafo_inverse_matrix} and the identity $\sigma_z(a)^\ast=\sigma_{\overline{z}}(a^\ast)$ we deduce
    \begin{equation*}
        \V(\Phi)=\sum_{j=1}^N \int_0^\infty\IL_{\sigma_{-i/4}(e^{-r\sigma_{-i/2}}(V_j))^\ast}\IR_{\sigma_{-i/4}(e^{-r\sigma_{-i/2}}(V_j))}\,dr.
    \end{equation*}
    Since maps of the form $\IL_{A^\ast}\IR_A$ are completely positive and positive linear combinations and limits of completely positive maps are again completely positive, it follows that $\V(\Phi)$ is completely positive.
    \item The identity $\W(\id_n)=\id_n$ is immediate from the definition, from which $\V(\id_n)=\id_n$ follows directly.
    % \item Let $(\Psi_t)$ be a quantum Markov semigroup such that $\Phi$ is its generator, i.e. 
    % \begin{equation}
    %     \Phi=\lim_{t\to 0}\frac{1}{t}(id_n-\Psi_t).
    % \end{equation}
    % Because the $\V$-transform is linear and $\V(id_n)=id_n$ by (iii), $\V(\Phi)$ is given by
    % \begin{equation}
    %     \V(\Phi)=\lim_{t\to 0}\frac{1}{t}(id_n-\V(\Psi_t)). 
    % \end{equation}
    % Since (ii) shows that $\V(\Psi_t)$ is completely positive, we can now conclude that 
    \qedhere
\end{enumerate}
\end{proof}

\subsection{Derivations for KMS-Symmetric Markov Generators on Matrix Algebras}
We are now in the position to prove the existence of a (twisted) derivation that implements the Dirichlet form associated with a KMS-symmetric quantum Markov semigroup on $M_n(\IC)$. We first present an abstract version that will later be generalized to quantum Markov semigroups on arbitrary von Neumann algebras. A more explicit version tailored for matrix algebras will be discussed below.

\begin{theorem}\label{thm:existence_derivation_matrix}
Let $(\Phi_t)$ be a KMS-symmetric quantum Markov semigroup on $M_n(\IC)$ and let $\L$ denote its generator. There exists a Hilbert space $\H$, a unital $\ast$-homorphism $\pi_l\colon M_n(\IC)\to B(\H)$, a unital $\ast$-antihomorphism $\pi_r\colon M_n(\IC)\to B(\H)$, and anti-linear isometric involution $\J\colon \H\to \H$ and a linear map $\delta\colon M_n(\IC)\to \H$ satisfying
\begin{enumerate}[(i)]
    \item $\pi_l(A)\pi_r(B)=\pi_r(B)\pi_l(A)$ for all $A,B\in M_n(\IC)$,
    \item $\J(\pi_l(A)\pi_r(B)\xi)=\pi_l(B)^\ast \pi_r(A)^\ast \J\xi$ for all $A,B\in M_n(\IC)$ and $\xi\in \H$,
    \item $\delta(A^\ast)=\J\delta(A)$ for all $A\in M_n(\IC)$,
    \item $\delta(AB)=\pi_l(\sigma_{-i/4}(A))\delta(B)+\pi_r(\sigma_{i/4}(B))\delta(A)$ and
    \item $\H=\mathrm{lin}\{\pi_l(A)\delta(B)|A,B\in M_n{\IC}\}$
\end{enumerate}
such that
\begin{equation} \label{eq:deriv_sqrt_of_gen_matrix}
    \langle A,\L(B)\rangle_\rho=\langle \delta(A),\delta(B)\rangle_\H
\end{equation}
for all $A,B\in M_n(\IC)$.
\end{theorem}
\begin{proof}
First, we claim that $\check{\L}$ is a KMS-symmetric conditionally completely negative map. The KMS-symmetry follows from Lemma \ref{lem:properties_V-trafo_matrix}(i). By Lemma \ref{lem:properties_V-trafo_matrix}(ii) $\check{\Phi}_t$ is completely positive for all $t\geq 0$. Since
\begin{equation*}
    \check{\L}=\lim_{t\to 0}\frac{1}{t}(\id_n-\check{\Phi}_t)
\end{equation*}
by Lemma \ref{lem:properties_V-trafo_matrix}(iii), it follows from the definitions that $\check{\L}$ is conditionally completely negative, proving the claim. We also observe that
\begin{equation*}
    \check{\L}(I_n)=2\int_0^\infty \sigma_{-i/4}e^{-r\sigma_{-i/2}}( \L( \sigma_{-i/4}e^{-r\sigma_{-i/2}}(I_n)))\,dr=0
\end{equation*}
by Proposition \ref{prop:V-trafo_inverse_matrix}.

% By \eqref{eq:rep_KMS_generator} there exists a KMS-symmetric completely positive map $\Psi$ on $M_n(\IC)$ such that 
% \begin{equation*}
%     \L=\IL_{(1+\sigma_{-i/2})^{-1}(\Psi(I_n))}+\IR_{(1+\sigma_{i/2})^{-1}(\Psi(I_n))}-\Psi.
% \end{equation*}

Next, we define a sesquilinear form $\ip{\cdot, \cdot}_\H$ on the algebraic tensor product $M_n(\IC)\odot M_n(\IC)$ by
\begin{equation*}
    \ip{A_1\otimes B_1,A_2\otimes B_2}_\H=-\frac{1}{2}\tr(B_1^\ast \rho^{1/2}\check{\L}(A_1^\ast A_2)\rho^{1/2}B_2).
\end{equation*}
Now consider the subspace 
\begin{equation*}
    N=\left\{\sum_j A_j\otimes B_j:\sum_j A_j\sigma_{-i/2}(B_j)=0\right\}.
\end{equation*}
Because $\check{\L}$ is completely conditionally negative, we see that for any $\sum_j A_j\otimes B_j\in N$ we have
\begin{align*}
    \sum_{jk}\ip{A_j\otimes B_j,A_k\otimes B_k}_\H&=-\frac{1}{2}\sum_{jk}\tr(B_j^\ast\rho^{1/2}\check{\L}(A_j^\ast A_k)\rho^{1/2}B_k)\\
    &=-\frac{1}{2}\sum_{jk}\tr(\rho^{1/2}\sigma_{-i/2}(B_j)^\ast\check{\L}(A_j^\ast A_k)\sigma_{-i/2}(B_k)\rho^{1/2})\\
    &\geq 0.
\end{align*}
Therefore, this sesquilinear form is positive semidefinite on $N$. Now let
\begin{equation*}
    \H = N/\{u \in N|\ip{u,u}_\H=0\}.
\end{equation*}
Then $\ip{\cdot,\cdot}$ induces an inner product on $\H$, which turns $\H$ into a Hilbert space (as $\H$ is finite-dimensional). We write $\sum_j A_j\otimes_\H B_j$ for the image of $\sum_j A_j\otimes B_j$ in $\H$ under the quotient map.

Define $\pi_l$ and $\pi_r$ by
\begin{equation*}
    \pi_l(X)\sum_jA_j\otimes_\H B_j=\sum_jXA_j\otimes_\H B_j,\,\pi_r(X)\sum_jA_j\otimes_\H B_j=\sum_jA_j\otimes_\H B_jX
\end{equation*}
for $X\in M_n(\IC)$ and $\sum_j A_j\otimes B_j\in N$. These maps are well defined. Indeed, they preserve $N$, and for all $u\in N$ with $\ip{u,u}_\H=0$ we have
\begin{equation*}
    \ip{\pi_l(X)\pi_r(Y)u,\pi_l(X)\pi_r(Y)u}=\ip{\pi_l(X^\ast X)\pi_r(YY^\ast)u,u}\leq 0
\end{equation*}
by Cauchy--Schwarz. From the definitions of $\pi_l$, $\pi_r$ and $\ip{\cdot,\cdot}_\H$ we now conclude that $\pi_l$ is a unital $\ast$-homomorphism and $\pi_r$ is a unital $\ast$-antihomomorphism. 

Subsequently, we will define the anti-linear isometric involution $\J$ on $\H$. Consider the map
\begin{equation*}
    A\otimes B\mapsto -B^\ast \otimes A^\ast
\end{equation*}
from $M_n(\IC)\odot M_n(\IC)$ to itself. This is an isometry since $\check{\L}$ is KMS-symmetric. Moreover, it preserves $N$. Consequently, it acts in a well-defined manner on the equivalence classes of $\H$, and we call this map $\J$. It is clear that $\J$ is an anti-linear involution. 

Lastly, we define the map $\delta: M_n(\IC) \to \H$ by
\begin{equation*}
    \delta(A)=\sigma_{-i/4}(A)\otimes_\H I_n-I_n\otimes\sigma_{i/4}(A). 
\end{equation*}
With this definition properties (i)-(iv) are immediate from the definitions. For property (v) note that
\begin{equation*}
    \sum_j A_j\otimes_\H B_j=\sum_j\pi_l(A_j)\delta(\sigma_{-i/4}(B_j))+A_j\sigma_{-i/2}(B_j)\otimes_\H I_n=\sum_j\pi_l(A_j)\delta(\sigma_{-i/4}(B_j))
\end{equation*}
for any $\sum_j A_j\otimes_\H B_j\in \H$. 

To complete the proof of the theorem, we need to show that \eqref{eq:deriv_sqrt_of_gen_matrix} holds. Let $A,B\in M_n(\IC)$. Then we have
\begin{align}
    \ip{\delta(A),\delta(B)}_\H=&\ \ip{\sigma_{-i/4}(A)\otimes_\H I_n-I_n\otimes_\H\sigma_{i/4}(A),\sigma_{-i/4}(B)\otimes_\H I_n-I_n\otimes_\H\sigma_{i/4}(B)}\nonumber\\
    =&\frac{1}{2}\Big(\tr(\sigma_{-i/4}(A^\ast)\rho^{1/2}\check{\L}(\sigma_{-i/4}(B))\rho^{1/2})+\tr(\rho^{1/2}\check{\L}(\sigma_{i/4}(A^\ast))\rho^{1/2}\sigma_{i/4}(B))\label{expr:deriv_inpr_nonzero_terms_matrix}\\
    &-\tr(\rho^{1/2}\check{\L}(\sigma_{i/4}(A^\ast)\sigma_{-i/4}(B))\rho^{1/2})-\tr(\sigma_{-i/4}(A^\ast)\rho^{1/2}\check{\L}(I_n)\rho^{1/2}\sigma_{i/4}(B))\Big).\label{expr:deriv_inpr_zero_terms_matrix}
\end{align}
The two terms in line \eqref{expr:deriv_inpr_zero_terms_matrix} are zero because $\check{\L}$ is KMS-symmetric and $\check{\L}(I_n)=0$. For the terms in line \eqref{expr:deriv_inpr_nonzero_terms_matrix} we use the KMS-symmetry of $\check{\L}$ and the fact that $\tr(\sigma_{it}(C)D)=\tr(C\sigma_{-it}(D))$ for all $t\in \IR$ and $C,D\in M_n(\IC)$ to conclude that
\begin{align*}
    \ip{\delta(A),\delta(B)}_\H=&\ \frac{1}{2}\tr(A^\ast\rho^{1/2}\sigma_{i/4}(\check{\L}(\sigma_{-i/4}(B)))\rho^{1/2})+\frac{1}{2}\tr(\rho^{1/2}A^\ast\rho^{1/2}\sigma_{-i/4}(\check{\L}(\sigma_{i/4}(B))))\\
    =&\ \tr(A^\ast\rho^{1/2}\W(\check{\L})(B)\rho^{1/2})\\
    =&\  \ip{A,\L(B)}_\rho,
\end{align*}
as desired.
\end{proof}

As a consequence of the previous result, we get a more explicit expression for the quadratic form associated with the generator of a KMS-symmetric quantum Markov semigroup on $M_n(\IC)$. An analogous expression can be found in \cite[Eq. (5.3)]{CM17}, \cite[Prop. 2.5]{CM20} for the special case of GNS-symmetric quantum Markov semigroups.

\begin{theorem}\label{thm:commutator_sqrt_matrix}
If $(\Phi_t)$ is a KMS-symmetric quantum Markov semigroup on $M_n(\IC)$ with generator $\L$, then there exist matrices $V_1,\dots,V_N\in M_n(\IC)$ such that $\{V_j\}_{j=1}^N=\{V_j^\ast\}_{j=1}^N$ and
\begin{equation*}
    \langle A,\L(B)\rangle_\rho=\sum_{j=1}^N \langle [V_j,A],[V_j,B]\rangle_\rho
\end{equation*}
for all $A,B\in M_n(\IC)$.
\end{theorem}
\begin{proof}
Let $\H$, $\pi_l$, $\pi_r$, $\J$ and $\delta$ be as in the last Theorem. The map $\pi_l\otimes \pi_r$ is a unital representation of $M_n(\IC)\otimes M_n(\IC)^\op$. It follows from the representation theory of matrix algebras \cite[Proposition IV.1.2.2.]{Bla06} that there exists an auxiliary Hilbert space $H$ such that $\H\cong M_n(\IC)\otimes H$, where $M_n(\IC)$ is endowed with the Hilbert--Schmidt  inner product, and 
\begin{align*}
\pi_l(A)(B\otimes \xi)&=AB\otimes \xi,\\
\pi_r(A)(B\otimes \xi)&=BA\otimes \xi
\end{align*}
for $A,B\in M_n(\IC)$ and $\xi\in H$ under this identification.

Moreover, since $\H$ is finite-dimensional by property (v), the space $H$ is finite-dimensional, say $H=\IC^N$.

Thus $\delta(A)=\sum_{j=1}^N \delta_j(A)\otimes e_j$ with the canonical orthonormal basis $(e_j)$ on $\IC^N$. It follows from property (iv) that
\begin{align*}
    \rho^{-1/4}\delta_j(AB)\rho^{-1/4}=A\rho^{-1/4}\delta_j(B)\rho^{-1/4}+\rho^{-1/4}\delta_j(A)\rho^{-1/4}B
\end{align*}
for $A,B\in M_n(\IC)$. In other words, $A\mapsto \rho^{-1/4}\delta_j(A)\rho^{-1/4}$ is a derivation. By the derivation theorem \cite[Theorem 9]{Kap53} there exists $V_j\in M_n(\IC)$ such that $\delta_j(A)=\rho^{1/4}[V_j,A]\rho^{1/4}$.

We conclude that
\begin{equation*}
    \langle A,\L(B)\rangle_\rho=\langle \delta(A),\delta(B)\rangle_\H=\sum_{j=1}^N \tr(\delta_j(A)^\ast \delta_j(B))=\sum_{j=1}^N \langle [V_j,A],[V_j,B]\rangle_\rho.
\end{equation*} 
Since $\ip{[V_j,A],[V_j,B]}_\rho=\ip{[V_j^\ast,B^\ast],[V_j^\ast,A^\ast]}_\rho$, we have
\begin{equation*}
    \ip{\delta(A),\delta(B)}_\H=\ip{\J\delta(B),\J\delta(A)}_\H=\sum_{j=1}^N\ip{[V_j,B^\ast],[V_j,A^\ast]}_\rho=\sum_{j=1}^N\ip{[V_j^\ast,A],[V_j^\ast,B]}_\rho.
\end{equation*}
This shows that $V_1,\dots, V_N$ can be chosen such that  $\{V_j\}_{j=1}^N=\{V_j^\ast\}_{j=1}^N$.
\end{proof}

\begin{remark}
    This theorem can also be proven without the use of Theorem \ref{thm:existence_derivation_matrix}. In the appendix we include a proof of Theorem \ref{thm:commutator_sqrt_matrix} using the $\V$-transform and the structure of generators of KMS-symmetric quantum Markov semigroups described in \cite[Theorem 4.4]{AC21}. 
\end{remark}

\section{\texorpdfstring{The $\V$-Transform}{The V-Transform}}\label{sec:V-transform}

As we have seen in \Cref{sec:matrix_algebras} in the case of matrix algebras, the key ingredient to construct the derivation associated with a KMS-symmetric quantum Markov semigroup is the $\V$-transform. This remains the case for semigroups on general von Neumann algebras, but the definition and properties of the $\V$-transform become much more delicate as it involves in general unbounded operators like the analytic generator of the modular group. In particular the fact that the $\V$-transform preserves completely positive maps, which is crucial for defining an inner product, requires new arguments in this general setting.

For technical convenience, we first define the $\V$-transform for bounded operators on the Hilbert space $L^2(M)$, before we transfer it to KMS-symmetric unital completely positive maps on $M$.

Throughout this section we fix a $\sigma$-finite von Neumann algebra $M$ and a faithful normal state $\phi$ on $M$. We also write $\phi$ for the element in $L^1(M)$ representing $\phi$ so that $\phi^{1/2}$ is the unique positive vector in $L^2(M)$ representing $\phi$, $\phi^{it}\cdot\phi^{-it}$ is the modular group etc. We write $\Delta$ and $J$ for the modular operator and modular conjugation associated with $\phi$, and $L^2_+(M)$ for the standard self-dual positive cone in $L^2(M)$.

\subsection{\texorpdfstring{$\V$-Transform of Bounded Operators on $L^2(M)$}{V-Transform of Bounded Operators on L²(M)}}

In this subsection we define the $\V$-transform on $B(L^2(M))$ and discuss some of its properties. The $\V$-transform in this setting is the map
\begin{equation*}
\V\colon B(L^2(M))\to B(L^2(M)),\,T\mapsto \check T=2\int_0^\infty \Delta^{1/4}e^{-r\Delta^{1/2}}T\Delta^{1/4}e^{-r\Delta^{1/2}}\,dr.
\end{equation*}

We first prove that it is well-defined. Note that $\Delta^{1/4}e^{-r\Delta^{1/2}}$ is bounded (and self-adjoint) for every $r>0$, so that the integrand is a bounded operator and the only difficulty is integrability.

\begin{lemma}
If $T\in B(L^2(M))$ and $\xi,\eta\in L^2(M)$, then
\begin{equation*}
r\mapsto \langle \xi,\Delta^{1/4}e^{-r\Delta^{1/2}}T\Delta^{1/4}e^{-r\Delta^{1/2}}\eta\rangle
\end{equation*}
is integrable on $(0,\infty)$, and
\begin{equation*}
2\left\lvert \int_0^\infty\langle\xi,\Delta^{1/4}e^{-r\Delta^{1/2}}T\Delta^{1/4}e^{-r\Delta^{1/2}}\eta\rangle\,dr\right\rvert\leq \norm{T}\norm{\xi}\norm{\eta}.
\end{equation*}
\end{lemma}
\begin{proof}
Let $E$ denote the spectral measure of $\Delta^{1/2}$. By the spectral theorem we have
\begin{align*}
&2\int_0^\infty \abs{\langle\xi,\Delta^{1/4}e^{-r\Delta^{1/2}}T\Delta^{1/4}e^{-r\Delta^{1/2}}\eta\rangle}\,dr\\
&\qquad\leq \norm{T}\int_0^\infty(\norm{\Delta^{1/4}e^{-r\Delta^{1/2}}\xi}^2+\norm{\Delta^{1/4}e^{-r\Delta^{1/2}}\eta}^2)\,dr\\
&\qquad=\norm{T}\int_0^\infty\int_0^\infty \lambda e^{-2\lambda r}\,d(\langle \xi,E(\lambda) \xi\rangle+\langle\eta,E(\lambda)\eta\rangle)\,dr\\
&\qquad=\norm{T}\int_0^\infty\lambda\int_0^\infty  e^{-2\lambda r}\,dr\,d(\langle \xi,E(\lambda) \xi\rangle+\langle\eta,E(\lambda)\eta\rangle)\\
&\qquad=\frac 1 2\norm{T}(\norm{\xi}^2+\norm{\eta}^2).
\end{align*}
Thus $r\mapsto \langle \xi,\Delta^{1/4}e^{-r\Delta^{1/2}}T\Delta^{1/4}e^{-r\Delta^{1/2}}\eta\rangle$ is integrable on $(0,\infty)$, and the desired inequality follows from the usual rescaling trick $\xi\mapsto \alpha\xi$, $\eta\mapsto \eta/\alpha$.
\end{proof}

\begin{definition}
For $T\in B(L^2(M))$ let $\check T$ denote the unique bounded linear operator on $L^2(M)$ such that
\begin{equation*}
\langle \xi,\check T\eta\rangle=2\int_0^\infty \langle \xi,\Delta^{1/4}e^{-r\Delta^{1/2}}T\Delta^{1/4}e^{-r\Delta^{1/2}}\eta\rangle\,dr
\end{equation*}
for all $\xi,\eta\in L^2(M)$.

We call the map
\begin{equation*}
    \V\colon B(L^2(M))\to B(L^2(M)),\,T\mapsto \check T
\end{equation*}
the $\V$-transform.
\end{definition}

Note that if $T$ commutes with the modular operator, then $\check T=T$.

\begin{proposition}
The $\V$-transform is a normal unital completely positive trace-preserving map on $B(L^2(M))$.
\end{proposition}
\begin{proof}
Let $E$ denote the spectral measure of $\Delta^{1/2}$. By the spectral theorem,
\begin{align*}
2\int_0^\infty\langle \Delta^{1/4}e^{-r\Delta^{1/2}}\xi,\Delta^{1/4}e^{-r\Delta^{1/2}}\eta\rangle\,dr&=2\int_0^\infty \int_0^\infty\lambda e^{-2\lambda r}\,d\langle \xi,E(\lambda)\eta\rangle\,dr\\
&=2\int_0^\infty\lambda\int_0^\infty e^{-2\lambda r}\,dr\,d\langle\xi,E(\lambda)\eta\rangle\\
&=\langle\xi,\eta\rangle
\end{align*}
for all $\xi,\eta\in L^2(M)$. Thus $\check 1=1$.

Since $\Delta^{1/4}e^{-r\Delta^{1/2}}$ is self-adjoint, the map $T\mapsto \Delta^{1/4}e^{-r\Delta^{1/2}} T \Delta^{1/4}e^{-r\Delta^{1/2}}$ is completely positive. Hence $\V$ is completely positive.

To prove that $\V$ is trace-preserving, let $T\in B(L^2(M))$ be positive. By the previous part, $\V(T)$ is positive and we have
\begin{align*}
    \tr(\V(T))&=2\int_0^\infty \tr(\Delta^{1/4}e^{-r\Delta^{1/2}}T\Delta^{1/4}e^{-r\Delta^{1/2}})\,dr\\
    &=2\int_0^\infty \tr(T^{1/2}\Delta^{1/2}e^{-2r\Delta^{1/2}}T^{1/2})\,dr\\
    &=\tr(T^{1/2}\V(1)T^{1/2})\\
    &=\tr(T).
\end{align*}
Interchanging the integral and the sum defining the trace is justified by Fubini's theorem since the integrand is positive in each case.

To prove that the $\V$-transform is normal, note first that it restricts to a bounded linear map $\V_\ast$ on the space of trace-class operators on $L^2(M)$ by the previous part. Moreover, if $S,T\in B(L^2(M))$ and $S$ is trace-class, then
\begin{equation*}
    \tr(\V_\ast(S)T)=2\int_0^\infty \tr(\Delta^{1/4}e^{-r\Delta^{1/2}}S\Delta^{1/4}e^{-r\Delta^{1/2}}T)\,dr=\tr(S\V(T)).
\end{equation*}
Thus $\V=(\V_\ast)^\ast$, which implies that $\V$ is normal.
\end{proof}

\begin{lemma}[Key property]\label{lem:key_V-transform}
\begin{enumerate}[(a)]
\item If $T\in B(L^2(M))$, then 
\begin{equation*}
\frac 1 2\langle\Delta^{1/4}\xi,\check T\Delta^{-1/4}\eta\rangle+\frac 1 2\langle \Delta^{-1/4}\xi,\check T\Delta^{1/4}\eta\rangle=\langle\xi,T\eta\rangle
\end{equation*}
for all $\xi,\eta\in\dom(\Delta^{1/4})\cap\dom(\Delta^{-1/4})$.

\item If $R\in B(L^2(M))$ such that
\begin{equation*}
\frac 1 2\langle\Delta^{1/4}\xi,R\Delta^{-1/4}\eta\rangle+\frac 1 2\langle \Delta^{-1/4}\xi,R\Delta^{1/4}\eta\rangle=\langle\xi,T\eta\rangle
\end{equation*}
for all $\xi,\eta\in\bigcap_{n\in\mathbb Z}\dom(\Delta^n)$, then $R=\check T$.
\end{enumerate}
\end{lemma}
\begin{proof}
(a) If $\xi,\eta\in \dom(\Delta^{1/4})\cap\dom(\Delta^{-1/4})$, then
\begin{align*}
&\langle \Delta^{1/4}\xi,\check T \Delta^{-1/4}\eta\rangle+\langle\Delta^{-1/4}\xi,\check T\Delta^{1/4}\eta\rangle\\
&\qquad=2\int_0^\infty(\langle \Delta^{1/2}e^{-r\Delta^{1/2}}\xi,Te^{-r\Delta^{1/2}}\eta\rangle+\langle e^{-r\Delta^{1/2}}\xi,T\Delta^{1/2}e^{-r\Delta^{1/2}}\eta\rangle)\,dr\\
&\qquad=-2\int_0^\infty \frac{d}{dr}\langle e^{-r\Delta^{1/2}}\xi,T e^{-r\Delta^{1/2}}\eta\rangle\,dr\\
&\qquad=2\langle\xi,T\eta\rangle.
\end{align*}
Here we used that since $\Delta^{1/2}$ is non-singular, $e^{-r\Delta^{1/2}}\zeta\to 0$ as $r\to\infty$ for every $\zeta\in L^2(M)$.

(b) If $\xi,\eta\in \bigcap_{n\in\mathbb Z}\dom(\Delta^n)$, then $\Delta^{1/4}e^{-r\Delta^{1/2}}\xi,\Delta^{1/4}e^{-r\Delta^{1/2}}\eta\in\dom(\Delta^{1/4})\cap\dom(\Delta^{-1/4})$ and hence
\begin{align*}
\langle\xi,\check T\eta\rangle&=2\int_0^\infty \langle \Delta^{1/4}e^{-r\Delta^{1/2}}\xi,T\Delta^{1/4}e^{-r\Delta^{1/2}}\eta\rangle\,dr\\
&=\int_0^\infty(\langle \Delta^{1/2}e^{-r\Delta^{1/2}}\xi,Re^{-r\Delta^{1/2}}\eta\rangle+\langle e^{-r\Delta^{1/2}}\xi,R\Delta^{1/2}e^{-r\Delta^{1/2}}\eta\rangle)\,dr.
\end{align*}
From here we conclude $\langle \xi,\check T\eta\rangle=\langle\xi,R\eta\rangle$ as above. Since $\bigcap_{n\in\mathbb Z}\dom(\Delta^n)$ is dense in $L^2(M)$, the equality $R=\check T$ follows.
\end{proof}

Informally, the identity from the previous lemma reads
\begin{equation*}
\frac 1 2\Delta^{1/4}\check T\Delta^{-1/4}+\frac 1 2\Delta^{-1/4}\check T\Delta^{1/4}=T.
\end{equation*}

\begin{proposition}\label{prop:properties_V-transform}
Let $T\in B(L^2(M))$.
\begin{enumerate}[(a)]
\item $\check T=\int_0^\infty \Delta^{-1/4}e^{-r\Delta^{-1/2}}T\Delta^{-1/4}e^{-r\Delta^{-1/2}}\,dr$ in the weak operator topology.
\item If $JT=TJ$, then $J\check T=\check T J$.
\item If $T\phi^{1/2}=\phi^{1/2}$, then $\check T\phi^{1/2}=\phi^{1/2}$.
\end{enumerate}
\end{proposition}
\begin{proof}
\begin{enumerate}[(a)]
\item Let $R=\int_0^\infty \Delta^{-1/4}e^{-r\Delta^{-1/2}}T\Delta^{-1/4}e^{-r\Delta^{1/2}}\,dr$. The existence is justified by the same arguments as for $\check T$. Replacing $\Delta$ by $\Delta^{-1}$ in \Cref{lem:key_V-transform}~(a) we obtain
\begin{align*}
\frac1  2\langle \Delta^{1/4}\xi,R\Delta^{-1/4}\eta\rangle+\frac 1 2\langle \Delta^{-1/4}\xi,R\Delta^{1/4}\eta\rangle=\langle\xi,T\eta\rangle
\end{align*}
for $\xi,\eta\in\dom(\Delta^{1/4})\cap\dom(\Delta^{-1/4})$. Now $R=\check T$ follows from \Cref{lem:key_V-transform}~(b).
\item One has $J\Delta^{1/4}=\Delta^{-1/4}J$, and by the spectral theorem also $Je^{-r\Delta^{1/2}}=e^{-r\Delta^{-1/2}}J$. Thus
\begin{align*}
J\check T&=\int_0^\infty \Delta^{-1/4}e^{-r\Delta^{-1/2}}J T\Delta^{1/4}e^{-r\Delta^{1/2}}\,dr\\
&=\int_0^\infty \Delta^{-1/4}e^{-r\Delta^{-1/2}}T J\Delta^{1/4}e^{-r\Delta^{1/2}}\,dr\\
&=\int_0^\infty \Delta^{-1/4}e^{-r\Delta^{-1/2}} T\Delta^{-1/4}e^{-r\Delta^{-1/2}}\,dr\,J.
\end{align*}
Now $J\check T=\check T J$ follows from (a).
\item This is immediate from the definition of $\check T$.\qedhere
\end{enumerate}
\end{proof}

We are particularly interested in the action of the $\V$-transform on Markov operators. Let us first recall the definition.

We call an operator $T\in B(L^2(M))$ \emph{positivity-preserving} if $T(L^2_+(M))\subset L^2_+(M)$, and \emph{completely positivity-preserving} if the amplification $T\otimes \mathrm{id}_{M_n(\mathbb C)}$ is positivity-preserving on $L^2(M)\otimes M_n(\mathbb C)\cong L^2(M_n(M))$ for every $n\in \mathbb N$. We use the terms ``positivity-preserving'' and ``completely positivity-preserving'' instead of positive and completely positive to avoid confusion with the concept of positive operators on a Hilbert space.

\begin{definition}
An operator $T\in B(L^2(M))$ is a \emph{Markov operator} if it is completely positivity-preserving and $T\phi^{1/2}=\phi^{1/2}$.
\end{definition}

To prove that the $\V$-transform also preserves Markov operators, we need a series of lemmas. Recall that the cones $P^\sharp$ and $P^\flat$ are defined as
\begin{equation*}
P^\sharp=\overline{\{x\phi^{1/2}\mid x\in M_+\}},\;P^\flat=\overline{\{\phi^{1/2}x\mid x\in M_+\}},
\end{equation*}
and further recall that $P^\sharp$ and $P^\flat$ are dual cones, that is, $\xi\in P^\sharp$ if and only if $\langle \xi,\eta\rangle\geq 0$ for all $\eta\in P^\flat$, and vice versa.

\begin{lemma}
Let $T\in B(L^2(M))$ be positivity-preserving. If $\xi \in P^\sharp$ (resp. $\xi\in P^\flat$) and $\eta\in P^\flat$ (resp. $\eta\in P^\sharp$), then
\begin{equation*}
\Re\langle\xi,\check T\eta\rangle\geq 0.
\end{equation*}
\end{lemma}
\begin{proof}
Let $x,y\in M_+$. By \Cref{lem:key_V-transform}~(a) we have
\begin{equation*}
\frac 1 2 \langle x\phi^{1/2},\check T(\phi^{1/2}y)\rangle+\frac 1 2\langle \phi^{1/2}x,\check T(y\phi^{1/2})\rangle=\langle \phi^{1/4}x\phi^{1/4},T(\phi^{1/4}y\phi^{1/4})\rangle\geq 0.
\end{equation*}
Since $T$ is positivity-preserving, it commutes with $J$. Hence so does $\check T$ by \Cref{prop:properties_V-transform}~(b). If we apply this to the first summand from the previous equation, we obtain
\begin{align*}
\langle x\phi^{1/2},\check T(\phi^{1/2}y)\rangle&=\langle J\check T(\phi^{1/2}y),J(x\phi^{1/2})\rangle=\langle \check T(y\phi^{1/2}),\phi^{1/2}x\rangle.
\end{align*}
Therefore
\begin{equation*}
\Re\langle x\phi^{1/2},\check T(\phi^{1/2}y)\rangle\geq 0.
\end{equation*}
This settles the claim for $\xi=x\phi^{1/2}$ and $\eta=\phi^{1/2}y$. For arbitrary $\xi\in P^\sharp$ and $\eta\in P^\flat$ the inequality follows by approximation. The proof for $\xi\in P^\flat$ and $\eta\in P^\sharp$ is analogous.
\end{proof}

\begin{lemma}
If $R\in B(L^2(M))$ such that
\begin{align*}
\Re\langle \xi,R\eta\rangle&\geq 0,\qquad \xi\in P^\sharp,\eta\in P^\flat,\\
\Re\langle \xi,R\eta\rangle&\geq 0,\qquad \xi\in P^\flat, \eta\in P^\sharp,
\end{align*}
then $\Re\langle \xi,R\eta\rangle\geq 0$ for all $\xi,\eta\in L^2_+(M)$.
\end{lemma}
\begin{proof}
Let $S=\{z\in\mathbb C\mid 0< \Re z< 1/2\}$ and
\begin{equation*}
f\colon \overline{S}\to \mathbb C,\,f(z)=e^{-\langle J\Delta^z(x\phi^{1/2}),R\Delta^z (y\phi^{1/2})\rangle}
\end{equation*}
for $x,y\in M_+$. The function $f$ is continuous on $\overline{S}$, holomorphic on $S$ and satisfies
\begin{align*}
\abs{f(z)}&=e^{-\Re\langle J\Delta^z(x\phi^{1/2}),R\Delta^z (y\phi^{1/2})\rangle}\\
&\leq e^{\norm{J\Delta^z(x\phi^{1/2})}\norm{R\Delta^z(y\phi^{1/2})}}\\
&\leq e^{\norm{R}\norm{\Delta^{\Re z}(x\phi^{1/2})}\norm{\Delta^{\Re z}(y\phi^{1/2})}}.
\end{align*}
By the spectral theorem,
\begin{equation*}
\norm{\Delta^{\Re z}(x\phi^{1/2})}^2\leq \norm{x\phi^{1/2}}^2+\norm{\Delta^{1/2}(x\phi^{1/2})}^2
\end{equation*}
and likewise for $y\phi^{1/2}$. Thus $f$ is bounded on $\overline{S}$.

If $\Re z=0$, then $\Delta^z(y\phi^{1/2})=\sigma^\phi_{\Im z}(y)\phi^{1/2}\in P^\sharp$ and $J\Delta^{z}(x\phi^{1/2})=\phi^{1/2}\sigma^\phi_{\Im z}(x)\in P^\flat$. Thus $\abs{f(z)}\leq 1$ by assumption. Similarly, $\abs{f(z)}\leq 1$ if $\Re z=1/2$.

It follows from the Phragmen--Lindelöf principle that $\abs{f(1/4)}\leq 1$, which means
\begin{equation*}
\Re\langle \Delta^{1/4}(x\phi^{1/2}),R\Delta^{1/4}(y\phi^{1/2})=\Re\langle J\Delta^{1/4}(x\phi^{1/2}),R\Delta^{1/4}(y\phi^{1/2})\rangle\rangle\geq 0.
\end{equation*}
This settles the claim for $\xi=\Delta^{1/4}(x\phi^{1/2})$, $\eta=\Delta^{1/4}(y\phi^{1/2})$ with $x,y\in  M_+$. For arbitrary $\xi,\eta\in L^2_+(M)$ the claim follows by approximation.
\end{proof}

\begin{lemma}
If $R\in B(L^2(M))$ such that $\Re\langle \xi,R\eta\rangle\geq 0$ for all $\xi,\eta\in L^2_+(M)$ and $JR=RJ$, then $R$ is positivity-preserving.
\end{lemma}
\begin{proof}
For $\xi,\eta\in L^2_+(M)$ we have $J\xi=\xi$, $J\eta=\eta$, and hence
\begin{equation*}
\langle R\eta,\xi\rangle=\langle J\xi,JR\eta\rangle=\langle \xi,R\eta\rangle.
\end{equation*}
Therefore $\langle \xi,R\eta\rangle$ is real and thus positive by assumption. As $L^2_+(M)$ is self-dual, the claim follows.
\end{proof}

\begin{proposition}\label{prop:V-trafo_Markov}
The $\V$-transform maps symmetric Markov operators to symmetric Markov operators.
\end{proposition}
\begin{proof}
If $T\in B(L^2(M))$ is positivity-preserving, it follows from the previous three lemmas that $\check T$ is positivity-preserving as well. If $T$ is completely positivity-preserving, the same argument applied to the amplifications $T\otimes\mathrm{id}_{M_n(\mathbb C)}$ shows that $\check T$ is completely positivity-preserving. That $T\phi^{1/2}=\phi^{1/2}$ implies $\check T\phi^{1/2}=\phi^{1/2}$ was established in \Cref{prop:properties_V-transform}~(c).
\end{proof}

\subsection{\texorpdfstring{$\V$-Transform of KMS-Symmetric Operators on $M$}{V-Transform of KMS-Symmetric Operators on M}}

Formally, the $\V$-transform of a bounded linear operator $\Phi$ on $M$ should be given by
\begin{equation*}
    2\int_0^\infty \sigma_{-i/4}e^{-r\sigma_{-i/2}}\Phi \sigma_{-i/4}e^{-r\sigma_{-i/2}}\,dr,
\end{equation*}
where we simply replaced the modular operator in the definition of $\V$ by the analytic generator of the modular group on $M$. However, it seems hard to make this formula rigorous. It is not even clear if $\sigma_{-i/2}$ generates a semigroup on $M$ in a suitable sense.

Instead, we take a different approach that relies on the correspondence between certain operators in $B(M)$ and operators in $B(L^2(M))$. This way we can only define the $\V$-transform for KMS-symmetric unital completely positive maps and generators of KMS-symmetric quantum Markov semigroups, but this suffices for our purposes.

Let us first recall the definition of a KMS-symmetric map in this setting. A linear map $\Phi\colon M\to M$ is called \emph{KMS-symmetric} (with respect to $\phi$) if 
\begin{equation*}
\langle J\Phi(x)^\ast J\phi^{1/2},y\phi^{1/2}\rangle=\langle Jx^\ast J \phi^{1/2},\Phi(y)\phi^{1/2}\rangle
\end{equation*}
for all $x,y\in M$. Using the trace-like functional $\tr$ on the Haagerup $L^1$ space, this can compactly be rewritten as
\begin{equation*}
    \tr(\Phi(x)^\ast \phi^{1/2}y\phi^{1/2})=\tr(x^\ast \phi^{1/2}\Phi(y)\phi^{1/2})
\end{equation*}
in analogy with the definition for matrix algebras.

There is the following correspondence between symmetric completely positivity-preserving maps on $L^2(M)$  and KMS-symmetric completely positive maps on $M$ \cite[Section 2]{Cip97}: If $\Phi$ is KMS-symmetric unital completely positive map on $M$, then there exists a unique bounded linear operator $T$ on $L^2(M)$ such that
\begin{equation*}
    T(\phi^{1/4}x\phi^{1/4})=\phi^{1/4}\Phi(x)\phi^{1/4}
\end{equation*}
for all $x\in M$. This operator is a symmetric Markov operator, which we denote by $\Phi^{(2)}$.

Conversely, if $T$ is a symmetric Markov operator on $L^2(M)$, then there exists a KMS-symmetric unital completely positive map $\Phi$ on $M$ such that $\Phi^{(2)}=T$.

If $\Phi$ is a KMS-symmetric unital completely positive map on $M$, then $\V(\Phi^{(2)})$ is a symmetric Markov operator by \Cref{prop:V-trafo_Markov}. This justifies the following definition.

\begin{definition}
Let $\Phi$ be a KMS-symmetric unital completely positive map on $M$. Its $\V$-transform $\check\Phi$ is the unique KMS-symmetric unital completely positive map $\Psi$ on $M$ such that
\begin{equation*}
\V(\Phi^{(2)})=\Psi^{(2)}.
\end{equation*}
We also write $\V(\Phi)$ for $\check\Phi$.
\end{definition}

The main object of interest of this article are semigroups of unital completely positive maps. While the $\V$-transform preserves unitality and complete positivity, it does not, in general, preserve the semigroup property. We will discuss in the following how one can still $\V$-transform generators of a class of such semigroups in a way that preserves complete positivity of the generated semigroup. We start by recalling some definitions.

A \emph{quantum Markov semigroup} is a family $(\Phi_t)_{t\geq 0}$ of normal unital completely positive maps on $M$ such that
\begin{itemize}
    \item $\Phi_0=\id_M$,
    \item $\Phi_{s+t}=\Phi_s\Phi_t$ for all $s,t\geq 0$,
    \item $\Phi_t\to \id_M$ in the pointwise weak$^\ast$ topology as $t\to 0$.
\end{itemize}
If $\Phi_t\to \id_M$ in operator norm as $t\to 0$, then $(\Phi_t)$ is called \emph{uniformly continuous}.

The generator of $(\Phi_t)$ is the weak$^\ast$ closed and densely defined operator $\L$ given by
\begin{align*}
    \dom(\L)&=\left\{x\in M : \lim_{t\to 0}\frac 1 t(x-\Phi_t(x))\text{ exists in the weak* topology}\right\}\\
    \L(x)&=\lim_{t\to 0}\frac 1 t(x-\Phi_t(x)) \text{ in the weak* topology}.
\end{align*}
The semigroup $(\Phi_t)$ is uniformly continuous if and only if the generator is bounded. The generators of uniformly continuous quantum Markov semigroups can be characterized as follows \cite[Theorem 14.7]{EL77a}: A bounded operator $\L$ on $M$ is the generator of a uniformly continuous quantum Markov semigroup if and only it is normal, $\L(1)=0$ and $\L$ is \emph{conditionally negative definite}, that is,
\begin{equation*}
    \sum_{j=1}^n y_j^\ast\L(x_j^\ast x_k)y_k\leq 0
\end{equation*}
whenever $x_1,\dots,x_n,y_1,\dots,y_n\in M$ with $\sum_{j=1}^n x_j y_j=0$.

Let $(\Phi_t)$ be a uniformly continuous KMS-symmetric Markov semigroup on $M$ and $(T_t)$ the associated symmetric Markov semigroup on $L^2(M,\phi)$. Let $\L$ and $\L_2$ denote the generator of $(\Phi_t)$ on $M$ and $(T_t)$ on $L^2(M,\phi)$, respectively. By the uniform continuity assumption, both are bounded linear operators. Thus we can form the $\V$-transform of $\L_2$, and the continuity of the $\V$-transform implies
\begin{equation*}
\check\L_2=\lim_{t\to 0}\frac 1 t(1-\check T_t)
\end{equation*}
in operator norm. However, since $(\check T_t)$ is not a semigroup in general, this does not imply directly that $\check \L_2$ generates a symmetric Markov semigroup. The lack of semigroup property can be taken care of by a suitable rescaling of the time parameter. More precisely, we have the following result.

\begin{proposition}
If $(T_t)$ is a symmetric Markov semigroup on $L^2(M,\phi)$, then the semigroup generated by $\check \L_2$ satisfies
\begin{equation*}
e^{-t\check \L_2}\xi=\lim_{n\to\infty}\check T_{t/n}^n\xi
\end{equation*}
for all $t\geq 0$ and $\xi\in L^2(M,\phi)$. In particular, $(e^{-t\check\L_2})$ is a symmetric Markov semigroup.
\end{proposition}
\begin{proof}
Since the $\V$-transform is a contraction, we have $\norm{\check T_t}\leq \norm{T_t}\leq 1$ for all $t\geq 0$. It follows from the Chernoff product formula [Engel, Nagel, Theorem III.5.2] that $e^{-t\check \L_2}\xi=\lim_{n\to\infty}\check T_{t/n}^n\xi$ for all $t\geq 0$ and $\xi\in L^2(M,\phi)$. Finally, as $\check T_s$ is a symmetric Markov operator for all $s\geq 0$, so is $e^{-t\check \L_2}$ for all $t\geq 0$.
\end{proof}

As a consequence of the previous theorem, there exists a unique KMS-symmetric quantum Markov semigroup $(\Psi_t)$ on $M$ such that
\begin{equation*}
\phi^{1/4}\Psi_t(x)\phi^{1/4}=e^{-t\check \L_2}(\phi^{1/4}x\phi^{1/4})
\end{equation*}
for all $x\in M$ and $t\geq 0$. Moreover, $(\Psi_t)$ is uniformly continuous. This justifies the following definition.

\begin{definition}
If $\L$ is the generator of a uniformly continuous KMS-symmetric Markov semigroup on $M$, its $\V$-transform $\check \L$ is the unique normal linear operator on $M$ such that
\begin{equation*}
\phi^{1/4}\check \L(x)\phi^{1/4}=\check\L_2(\phi^{1/4}x\phi^{1/4})
\end{equation*}
for all $x\in M$.
\end{definition}

By the discussion above, $\check \L$ is again the generator of a KMS-symmetric quantum Markov semigroup on $M$. Note moreover that since $\L(1)=0$, the only completely positive generator of a quantum Markov semigroup is $\L=0$. Therefore there is no conflict between our definitions of the $\V$-transform of completely positive maps and Markov generators.

\section{Derivations for Uniformly Continuous Quantum Markov Semigroups}\label{sec:uniformly_bounded}

In this section we study the correspondence between uniformly KMS-symmetric quantum Markov semigroups and certain twisted derivations with values in bimodules. We show that -- like in the case of matrix algebras -- every uniformly continuous KMS-symmetric quantum Markov semigroup gives rise to a derivation (\Cref{thm:existence:deriv_bounded}) and that this derivation is unique (\Cref{thm:uniqueness}).

Throughout this section let $M$ be a von Neumann algebra and $\phi$ a faithful normal state on $M$. KMS symmetry is always understood with respect to this state $\phi$ and multiplication of elements in $L^2(M)$ is understood as the multiplication induced by the full left Hilbert algebra associated with $\phi$. For a left-bounded vector $a\in L^2(M)$ we write $\pi_l(a)$ for the bounded operator of left multiplication by $a$. Likewise, if $a\in L^2(M)$ is right-bounded, we write $\pi_r(a)$ for the right multiplication operator.

\subsection{Existence and Innerness}

After we established the existence and properties of the $\V$-transform in the previous section, the proof for the existence of a derivation associated with a uniformly continuous KMS-symmetric quantum Markov semigroup follows the same strategy as in the finite-dimensional case. 

To establish that the derivation is inner, we need the following result, which is an easy consequence of the Christensen--Evans theorem \cite{CE79}. Recall that a bounded linear map between $C^\ast$-algebras is called \emph{decomposable} if it is a linear combination of completely positive maps \cite{Haa85}.

\begin{proposition}\label{prop:cnd_decomposable}
Every generator of a uniformly continuous quantum Markov semigroup on $M$ is the difference of two normal completely positive maps. In particular, it is decomposable.
\end{proposition}
\begin{proof}
Let $\L$ be a normal conditionally negative definite map on $M$. By \cite[Theorem 3.1]{CE79} there exists $k\in M$ and a completely positive map $\Phi\colon M\to M$ such that
\begin{equation*}
    \L(x)=k^\ast x+xk-\Phi(x)
\end{equation*}
for all $x\in M$. Since $\L$ is normal, so is $\Phi$. It follows from \cite[Proposition 6.10]{Pis20} that $\L$ is the difference of two normal completely positive maps.
\end{proof}

For the following result, let us recall the notion of correspondences (see \cite[Chapter 5, Appendix B]{Con94} for example). An \emph{$M$-$M$-correspondence} or simply \emph{correspondence} in our case is a Hilbert space $\H$ endowed with normal unital $\ast$-homomorphisms $\pi_l^\H\colon M\to B(\H)$, $\pi_r^\H\colon M^\op\to B(\H)$ such that $\pi_l^\H(M)$ and $\pi_r^\H(M^\op)$ commute. We write $x\cdot\xi\cdot y$ or simply $x\xi y$ for $\pi_l^\H(x)\pi_r^\H(y^\op)\xi$.

Every correspondence gives rise to a representation of the \emph{binormal tensor product} $M\otimes_\bin M^\op$, which is defined as follows (see \cite{EL77b}). A linear functional $\omega$ from the algebraic tensor product $M\odot M^\op$ to $\IC$ is called a \emph{binormal state} if $\omega(u^\ast u)\geq 0$ for all $u\in M\odot M^\op$, $\omega(1)=1$ and the maps $x\mapsto \omega(x\otimes y_0^\op)$, $y^\op\mapsto \omega(x_0\otimes y^\op)$ are weak$^\ast$ continuous for every $x_0,y_0\in M$. For $u$ in the algebraic tensor product $M\odot M^\op$ let
\begin{equation*}
    \norm{u}_\bin=\sup\{\omega(u^\ast u)^{1/2}\mid \omega\text{ binormal state on } M\odot M^\op\}.
\end{equation*}
The binormal tensor product $M\otimes_\bin M^\op$ is the completion of $M\odot M^\op$ with respect to the norm $\norm\cdot_\op$.

%For the following result note that if $a\in M\phi^{1/2}$ and $b\in \phi^{1/2}M$, then $\pi_l(a)\in M$ and $\pi_r(b)\in M^\prime$, hence $J\pi_r(b)^\ast J\in M$.

\begin{theorem}\label{thm:existence:deriv_bounded}
Let $(\Phi_t)$ be a uniformly continuous KMS-symmetric quantum Markov semigroup on $M$ and let $\L$ denote its generator.  There exists a correspondence $\H$, an anti-unitary involution $\J\colon \H\to \H$ and a bounded operator $\delta\colon L^2(M)\to \H$ satisfying
\begin{enumerate}[\indent (a)]
    \item $\J(x\xi y)=y^\ast(\J\xi)x^\ast$ for all $x,y\in M$ and $\xi\in \H$,
    \item $\delta(Ja)=\J\delta(a)$ for all $a\in L^2(M)$,
    \item $\delta(ab)=\pi_l(a)\cdot\delta(b)+\delta(a)\cdot J\pi_r(b)^\ast J$ for all $a\in M\phi^{1/2}$, $b\in \phi^{1/2}M$,
    \item $\overline{\mathrm{lin}}\{\delta(a)x\mid a\in L^2(M),\,x\in M\}=\H$
\end{enumerate}
such that
\begin{equation*}
    \L_2=\delta^\ast \delta.
\end{equation*}
Moreover, there exists $\xi_0\in \H$ such that 
\begin{equation*}
    \delta(a)=\pi_l(a)\cdot\xi_0-\xi_0\cdot J\pi_r(a)^\ast J
\end{equation*}
for $a\in M\phi^{1/2}\cap \phi^{1/2}M$.
\end{theorem}
\begin{proof}
Let
\begin{equation*}
    \omega\colon M\odot M^\op\to \IC,\,x\otimes y^\op\mapsto -\frac 1 2\langle \phi^{1/2},\check \L(x)\phi^{1/2}y\rangle.
\end{equation*}
We first show that $\omega$ extends continuously to $M\otimes_\bin M^\op$.

Since $\check\L$ is the generator of a uniformly continuous quantum Markov semigroup, by \Cref{prop:cnd_decomposable} there exist normal completely positive maps $\Psi_1,\Psi_2$ on $M$ such that $\check\L=\Psi_2-\Psi_1$. For $j\in \{1,2\}$, the maps
\begin{equation*}
    \omega_j\colon M\odot M^\op\to \IC,\,x\otimes y^\op\mapsto \frac 1 2\langle \phi^{1/2},\Psi_j(x)\phi^{1/2}y\rangle
\end{equation*}
are positive and separately weak$^\ast$ continuous. Hence they extend to positive functionals on $M\otimes_\bin M^\op$ by definition of $\norm\cdot_\bin$. Thus $\omega$ extends to a bounded linear map on $M\otimes_\bin M^\op$. We continue to denote the extension of $\omega$ to $M\otimes_\bin M^\op$ by $\omega$.

Let
\begin{equation*}
    q\colon M\odot M^\op\to L^2(M),\,x\otimes y^\op\mapsto x\phi^{1/2}y.    
\end{equation*}
 Clearly, the kernel of $q$ is a left ideal of $M\odot M^\op$. Let $I$ denote its closure, which is a closed left ideal of $M\otimes_\bin M^\op$.

 If $u=\sum_{j=1}^n x_j\otimes y_j^\op\in \ker q$, then
 \begin{align*}
    \omega(u^\ast u)&=-\frac 1 2\sum_{j,k=1}^n \langle \phi^{1/2}y_j,\check \L(x_j^\ast x_k)\phi^{1/2}y_k\rangle\\
    &=\lim_{t\to 0}\frac  1 {2t} \left(\sum_{j,k=1}^n \langle \phi^{1/2}y_j,\check \Phi_t(x_j^\ast x_k)\phi^{1/2}y_k\rangle-\norm{q(u)}^2 \right)\\
    &=\lim_{t\to 0}\frac  1 {2t}\sum_{j,k=1}^n \langle \phi^{1/2}y_j,\check \Phi_t(x_j^\ast x_k)\phi^{1/2}y_k\rangle.
 \end{align*}
 The last expression is positive since $\check \Phi_t$ is completely positive for all $t\geq 0$. By continuity, it follows that $\omega(u^\ast u)\geq 0$ for all $u\in I$.

 Let $\H$ be the GNS Hilbert space associated with $\omega\vert_I$ and $\pi_\omega$ the GNS representation of $M\otimes_\bin M^\op$ on $\H$, that is, $\H$ is the completion of $I$ with respect to the inner product $(x,y)\mapsto \omega(x^\ast y)$ and $\pi_\omega(x)[y]_\H=[xy]_\H$, were $[\cdot]_\H$ denotes the canonical map $I\to \H$. As $\omega$ is separately weak$^\ast$ continuous, it follows easily that the actions $x\mapsto \pi_\omega(x\otimes 1)$ and $y^\op\mapsto \pi_\omega(1\otimes y^\op)$ are normal. These actions make $\H$ into an $M$-$M$-correspondence.

 Moreover, let $(e_\lambda)$ be a right approximate identity for $I$ consisting of positive contractions. Since $\norm{[e_\lambda]_\H}=\omega(e_\lambda^2)^{1/2}\leq \norm{\omega}^{1/2}$, we may assume additionally that $[e_\lambda]_\H$ converges weakly to some vector $\xi_\omega\in \H$. If $u\in I$, then
\begin{equation*}
     [u]_\H=[\lim_\lambda u e_\lambda]=\lim_\lambda \pi_\omega(u)[e_\lambda]_\H=\pi_\omega(u)\xi_\omega.
\end{equation*}
In particular, $\xi_\omega$ is a cyclic vector for $\pi_\omega$.

To define $\J$, first define an anti-linear map $\J_0$ on $M\odot M^\op$ by $\J_0(x\otimes y^\op)=y^\ast \otimes (x^\ast)^\op$. A direct computation shows $q\J_0=Jq$. In particular, $\J_0$ leaves $\ker q$ invariant.

Furthermore, if $x_1,x_2,y_1,y_2\in M$, then
\begin{align*}
    \omega(\J_0(x_1\otimes y_1^\op)^\ast\J_0(x_2\otimes y_2^\op))&=-\frac 1 2\langle \phi^{1/2},\check\L(y_1 y_2^\ast)\phi^{1/2}x_2^\ast x_1\rangle\\
    &=-\frac 1 2\langle \phi^{1/2}x_1^\ast x_2,\check \L(y_1 y_2^\ast)\phi^{1/2}\rangle\\
    &=-\frac 1 2\langle \phi^{1/2}\check\L(x_1^\ast x_2),y_1 y_2^\ast\phi^{1/2}\rangle\\
    &=-\frac 1 2\langle \phi^{1/2},y_1 y_2^\ast \phi^{1/2}\check\L(x_2^\ast x_1)\rangle\\
    &=-\frac 1 2\langle J(y_1 y_2^\ast \phi^{1/2}\check\L(x_2^\ast x_1)),J\phi^{1/2}\rangle\\
    &=-\frac 1 2\langle \check \L(x_1^\ast x_2)\phi^{1/2}y_2 y_1^\ast,\phi^{1/2}\rangle\\
    &=-\frac 1 2\langle \phi^{1/2},\check\L(x_2^\ast x_1)\phi^{1/2}y_1 y_2^\ast\rangle\\
    &=\omega((x_2\otimes y_2^\op)^\ast (x_1\otimes y_1^\op)).
\end{align*}

Therefore the map
\begin{equation*}
    [I]_\H\to [I]_\H,\,[x\otimes y^\op]_\H\mapsto [y^\ast\otimes (x^\ast)^\op]_\H    
\end{equation*}
extends to an isometric anti-linear operator $\J$ on $\H$. Obviously, $\J$ is an involution and property (a) follows directly from the definition.

Finally let us construct $\delta$. For $a\in M\phi^{1/2}\cap \phi^{1/2}M$ let 
\begin{equation*}
     \delta(a)=\pi_\omega(\pi_l(a)\otimes 1-1\otimes (J\pi_r(a)^\ast J)^\op)\xi_\omega.
\end{equation*}
Note that
\begin{equation*}
     q(\pi_l(a)\otimes 1-1\otimes (J\pi_r(a)^\ast J)^\op)=\pi_l(a)\phi^{1/2}-\pi_r(a)\phi^{1/2}=0
\end{equation*}
so that $\pi_l(a)\otimes 1-1\otimes (J\pi_r(a)^\ast J)^\op\in I$ and $\delta(a)=[\pi_l(a)\otimes 1-1\otimes (J\pi_r(a)^\ast J)^\op]_\H$.
 
We have
\begin{align*}
    \norm{\delta(a)}_\H^2&=\norm{[\pi_l(a)\otimes 1-1\otimes (J\pi_r(a)^\ast J)^\op]_\H}^2\\
    &=\omega(\abs{\pi_l(a)\otimes 1-1\otimes (J\pi_r(a)^\ast J)^\op}^2)\\
    &=\omega(\pi_l(a^\sharp a)\otimes 1)+\omega(1\otimes (J\pi_r(aa^\flat)J)^\op)\\
    &\quad-\omega(\pi_l(a^\sharp)\otimes (J\pi_r(a)^\ast J)^\op)-\omega(\pi_l(a)\otimes (J\pi_r(a^\flat)^\ast J)^\op)\\
    &=-\frac 1 2 \langle \phi^{1/2},\check \L_2(\Delta^{1/4}(a^\sharp a))\rangle-\frac 1 2\langle\Delta^{1/4}J (aa^\flat),\check \L_2\phi^{1/2}\rangle\\
    &\quad  +\frac 1 2\langle \Delta^{1/4}J a,\check \L_2\Delta^{1/4}(a^\sharp)\rangle+\frac 1 2\langle \Delta^{1/4}J a^\flat,\check \L_2\Delta^{1/4}a\rangle\\
    &\overset{(1)}{=}\frac 1 2\langle J\Delta^{-1/4}a,\check \L_2 J\Delta^{1/4}a\rangle+\frac 1 2\langle \Delta^{-1/4}a,\check \L_2\Delta^{1/4}a\rangle\\
    &\overset{(2)}{=}\frac 1 2\langle \Delta^{1/4}a,\check \L_2\Delta^{-1/4}a\rangle+\frac 1 2\langle \Delta^{-1/4}a,\check \L_2\Delta^{1/4}a\rangle\\
    &\overset{(3)}{=}\langle a,\L_2(a)\rangle.
\end{align*}
Here we used the symmetry of $\check\L_2$ and $\check\L_2\phi^{1/2}=0$ for (1), the symmetry of $\check\L_2$ and $\check \L_2 J=J\check \L_2$ for (2) and the key property from \Cref{lem:key_V-transform} for (3).

Therefore the map $\delta$ extends to a bounded linear operator from $L^2(M)$ to $\H$, and this extension, still denoted by $\delta$, satisfies $\delta^\ast \delta=\L_2$. Clearly,
\begin{equation*}
    \delta(ab)=\pi_l(a)\cdot\delta(b)+\delta(a)\cdot J\pi_r(b)^\ast J     
\end{equation*}
for $a,b\in M\phi^{1/2}\cap \phi^{1/2}M$. If we only have $a\in M\phi^{1/2}$ and $b\in \phi^{1/2}M$, a standard approximation argument \cite[Lemma 1.3]{Haa75} shows that this identity continues to hold, which settles property (c).

Property (b) is clear from the definition if $a\in M\phi^{1/2}\cap \phi^{1/2}M$, and can be extended to $a\in L^2(M)$ again by approximation.

Finally, if $u=\sum_{j=1}^n x_j\otimes y_j^\op\in \ker q$ and $x_j$ is analytic for $\sigma^\phi$ for all $j\in \{1,\dots,n\}$, then 
\begin{align*}
    u&=\sum_{j=1}^n x_j\otimes y_j^\op\\
    &=\sum_{j=1}^n (x_j\otimes y_j^\op-1\otimes (\sigma^\phi_{i/2}(x_j)y_j)^\op),
\end{align*}
where we used $0=q(u)=\phi^{1/2}\sum_{j=1}^n \sigma^\phi_{i/2}(x_j)y_j$. Thus
\begin{equation*}
    [u]_\H=\sum_{j=1}^n \delta(x_j\phi^{1/2})y_j.
\end{equation*}
Since $\pi_\omega$ is normal and $[\ker q]_\H$ is dense in $\H$ by definition, property (d) follows.
\end{proof}

\begin{remark}
If one only wants to show the existence of the derivation $\delta$, one could work directly with the GNS representation of $\omega$ on $\ker q$ without passing to the binormal tensor product. Hence \Cref{prop:cnd_decomposable} and thus the Christensen--Evans theorem is only needed to show that the derivation is inner.
\end{remark}

\subsection{Uniqueness}

We show next that the triple $(\H,\J,\delta)$ constructed in \Cref{thm:existence:deriv_bounded} is uniquely determined by the semigroup $(\Phi_t)$ up to isomorphism. Let us first introduce some terminology for triples of this kind.

\begin{definition}
We call a pair $(\H,\J)$ consisting of an $M$-$M$-correspondence and an anti-unitary involution $\J\colon \H\to \H$ a \emph{self-dual $M$-$M$-correspondence} if
\begin{equation*}
    \J(x\xi y)=y^\ast (\J\xi) x^\ast
\end{equation*}
for all $x,y\in M$ and $\xi\in \H$.

We call a triple $(\H,\J,\delta)$ consisting of a self-dual $M$-$M$-correspondence $(\H,\J)$ and a closed  operator $\delta\colon \dom(\delta)\subset L^2(M)\to \H$ a \emph{first-order differential calculus} if 
\begin{enumerate}[\indent (a)]
    \item $J\dom(\delta)=\dom(\delta)$ and $\delta(Ja)=\J\delta(a)$ for all $a\in L^2(M)$,
    \item Whenever $a\in \dom(\delta)\cap M\phi^{1/2}$, $b\in \dom(\delta)\cap \phi^{1/2}M$, then $ab\in \dom(\delta)$ and $\delta(ab)=\pi_l(a)\delta(b)+\delta(a)J\pi_r(b)^\ast J$,
    \item $\overline{\mathrm{lin}}\{\delta(a)x\mid a\in L^2(M),\,x\in M\}=\H$.
\end{enumerate}
\end{definition}

With this definition, \Cref{thm:existence:deriv_bounded} says that for every bounded Markov generator $\L_2$ on $L^2(M)$ there exists a first-order differential calculus $(\H,\J,\delta)$ such that $\L_2=\delta^\ast \delta$. In this subsection we will show that $(\H,\J,\delta)$ is uniquely determined by $\L_2$.

To lighten the notation, we write $a\xi$ for $\pi_l(a)\xi$ if $a\in M\phi^{1/2}$ and $\xi b$ for $\xi\cdot J\pi_r(b)^\ast J$ if $b\in M\phi^{1/2}$ for the remainder of this subsection.

The first step towards uniqueness is a purely algebraic consequence of the properties (a) and (b) of a first-order differential calculus.

\begin{lemma} \label{lem:uniqueness_algebraic_calculation}
If $(\H,\J,\delta)$ is a first-order differential calculus  and $\delta$ is bounded, then
\begin{align*}
    \ip{\delta(\Delta^{1/4}&a)\Delta^{1/4}b,\delta(\Delta^{-1/4}c)}_\H+\ip{\delta(\Delta^{-1/4}a)\Delta^{-1/4}b,\delta(\Delta^{1/4}c)}_\H\\
    =&\ip{\delta(\Delta^{-1/4}(ab)),\delta(\Delta^{1/4}c)}_\H+\ip{\delta(\Delta^{1/4}a),\delta(\Delta^{-1/4}cJ(\Delta^{-1/4}b))}_\H\\
    &-\ip{\delta(J(\Delta^{1/4}c)\Delta^{1/4}a),\delta(J(\Delta^{-1/4}b))}_\H
\end{align*}
for all $a,b,c\in M_\phi^a\phi^{1/2}$, where $M^a_\phi$ denotes the set of all entire analytic elements for $\sigma^\phi$.
\end{lemma}
\begin{proof}
As $a,b,c\in M_\phi^a\phi^{1/2}$, these elements lie in $\bigcap_{n\in\mathbb Z}\dom(\Delta^n)$, and arbitrary powers of the modular operator map them to left- and right-bounded vectors. In particular, all expressions in the claimed equation are well-defined.

Using properties (a) and (b) of a first-order differential calculus, we can do the following computations. Let $x,y,z \in M_\phi^a\phi^{1/2}$. Then we have
\begin{align*}
    \ip{\delta(x)y,\delta(z)}_\H&=\ip{\delta(x)y,\delta(z)}_\H+\ip{x\delta(y),\delta(z)}_\H-\ip{\J(\delta(z)),\J(x\delta(y))}_\H\\
    &=\langle \delta(xy),\delta(z)\rangle_\H-\langle\delta(Jz),\delta(Jy)Jx\rangle_\H\\
    &=\ip{\delta(xy),\delta(z)}_\H-\ip{\delta(Jz)\Delta^{1/2}x,\delta(Jy)}_\H
\end{align*}
and
\begin{align*}
    \ip{\delta(x)y,\delta(z)}_\H&=\ip{\delta(x),\delta(z)J\Delta^{-1/2}y}_\H+\ip{\delta(x),z\delta(J\Delta^{-1/2}y)}_\H-\ip{\delta(x),z\delta(J\Delta^{-1/2}y)}_\H\\
    &=\ip{\delta(x),\delta(zJ\Delta^{-1/2}y)}_\H-\ip{\delta(x),z\delta(J\Delta^{-1/2}y)}_\H\\
    &=\ip{\delta(x),\delta(zJ\Delta^{-1/2}y)}_\H-\ip{J(\Delta^{1/2}z)\delta(x),\delta(J\Delta^{-1/2}y)}_\H.
\end{align*}
If we take $x=\Delta^{-1/4}a$, $y=\Delta^{-1/4}b$, $z=\Delta^{1/4}c$ in the first identity and $x=\Delta^{1/4}a$, $y=\Delta^{1/4}b$, $z=\Delta^{-1/4}c$ in the second identity and add them up, we obtain
\begin{align*}
    &\ip{\delta(\Delta^{-1/4}a)\Delta^{-1/4}b,\delta(\Delta^{1/4}c)}_\H+\ip{\delta(\Delta^{1/4}a)\Delta^{1/4}b,\delta(\Delta^{-1/4}c)}_\H\\
    &=\langle \delta(\Delta^{-1/4}(ab)),\delta(\Delta^{1/4}c)\rangle_\H-\langle\delta(J\Delta^{1/4}c)\Delta^{1/4}a,\delta(J\Delta^{-1/4}b)\rangle_\H\\
    &\quad +\langle\delta(\Delta^{1/4}a),\delta((\Delta^{-1/4}c)J\Delta^{-1/4}b)\rangle_\H-\langle (J\Delta^{1/4}c)\delta(\Delta^{1/4}a),\delta(J\Delta^{-1/4}b)\rangle_\H\\
    &=\langle \delta(\Delta^{-1/4}(ab)),\delta(\Delta^{1/4}c)\rangle_\H+\langle\delta(\Delta^{1/4}a),\delta((\Delta^{-1/4}c)J\Delta^{-1/4}b)\rangle_\H\\
    &\quad-\langle \delta((J\Delta^{1/4}c)\Delta^{1/4}a),\delta(J\Delta^{-1/4}b)\rangle_\H,
\end{align*}
where we used again property (b) of a first-order differential calculus in the last step.
\end{proof}

The significance of this result is that the right side depends only on the inner product of elements from the range of $\delta$ and not the bimodule generated by the range. If the modular operator $\Delta$ is trivial, that is, $\phi$ is a trace, one can conclude uniqueness of the derivation directly from this lemma. In the general case, substantially more work is needed. In particular, there are analytical difficulties that are absent in the case of tracially symmetric (or more generally GNS-symmetric) quantum Markov semigroups.

One tool we use are spectral subspaces of the analytic generator of the modular group. We recall the definition and some of their properties here. See \cite{CZ76} for more details.

For $0<\lambda_1<\lambda_2$ let
\begin{equation*}
    M[\lambda_1,\lambda_2]=\left\{x\in \bigcap_{t\in\IR}\dom(\sigma^\phi_{it}): \varlimsup_{t\to\infty}\norm{\sigma^\phi_{it}(x)}^{1/t}\leq \frac 1{\lambda_1},\,\varlimsup_{t\to\infty}\norm{\sigma^\phi_{-it}(x)}^{1/t}\leq \lambda_2\right\}.
\end{equation*}
This is a norm closed subspace of $M$, invariant under $\sigma^\phi$, and the spectrum of the restriction of $\sigma^\phi_{-i}$ to $M[\lambda_1,\lambda_2]$ is contained in $[\lambda_1,\lambda_2]$ (see \cite[(iii)--(v), p. 351]{CZ76}). Moreover, the union $\bigcup_{0<\lambda_1\leq \lambda_2<\infty}M[\lambda_1,\lambda_2]$ is weak$^\ast$ dense in $M$ \cite[(vi), p. 356]{CZ76}. Additionally, $\sigma^\phi_t$ is given by $e^{itH}$ for some $H\in B(M[\lambda_1,\lambda_2])$ with $\spec(H)\subset[- \ln(\lambda_1),\ln(\lambda_2)]$ \cite[Theorem 5.2, p. 349]{CZ76}.

\begin{lemma} \label{lem:tens_prod_mod_op}
    Let $0<\lambda_1<\lambda_2$. Define $X$ to be the completion of $(\1_{[\lambda_1,\lambda_2]}(\Delta)L^2(M))\odot M[\lambda_1,\lambda_2]$ with the projective cross norm. Then the bounded operator $T: X\to X$ defined on pure tensors by
    \begin{equation*}
        T(\eta \otimes x) = \Delta^{1/4}(\eta)\otimes \sigma^\phi_{-i/4}(x)
    \end{equation*}
    is well defined and $\spec(T)\subset(0,\infty)$.
\end{lemma}

\begin{proof}
    The spectrum of $\Delta^{1/4}$ restricted to $\1_{[\lambda_1,\lambda_2]}(\Delta)L^2(M)$ is contained in $[\lambda_1^{1/4},\lambda_2^{1/4}]$ by definition. Since the restriction of $\sigma^\phi_{-i/4}$ to $M[\lambda_1,\lambda_2]$ is $e^{1/4 H}$ for some $H\in B(M[\lambda_1,\lambda_2])$ with $\spec(H)\subset[-\ln(\lambda_1),\ln(\lambda_2)]$, we know that the spectrum of the restriction of $\sigma^\phi_{-i/4}$ to $M[\lambda_1,\lambda_2]$ is also contained in $[\lambda_1^{1/4},\lambda_2^{1/4}]$. Then $\Delta^{1/4}\otimes\id$ and $\id\otimes \sigma^\phi_{-i/4}$ are well defined and have spectra contained in $[\lambda_1^{1/4},\lambda_2^{1/4}]$, so $\Delta^{1/4}\otimes \sigma^\phi_{-i/4}$ has spectrum contained in $(0,\infty)$ \cite[p. 96]{Sch69}. 
\end{proof}

\begin{theorem} \label{thm:uniqueness}
    Let $(\H_1,\J_1,\delta_1)$ and $(\H_2,\J_2,\delta_2)$ be first order differential calculi for $M$ such that $\delta_1$ and $\delta_2$ are bounded and $\delta_1^\ast\delta_1=\delta_2^\ast\delta_2$. Then there exists a unitary bimodule map $\Theta: \H_1\to \H_2$ intertwining $\J_1$ and $\J_2$ such that $\Theta(\delta_1(a))=\delta_2(a)$ for all $a\in L^2(M)$. 
\end{theorem}

\begin{proof}
    The unitary bimodule map $\Theta$ will be given by
    \begin{equation*}
        \Theta(\delta_1(a)b)=\delta_2(a)b
    \end{equation*}
    on elements of the form $\delta_1(a)b$ with $a,b\in L^2(M)$. The difficult part of the proof is to show that this map is isometric; the other properties will follow naturally.

    Let $0<\lambda_1<\lambda_2$ be arbitrary. Let $X$ and $T$ be as in Lemma \ref{lem:tens_prod_mod_op}. Note that $T$ is invertible since $0\notin \spec(T)$. On $(\1_{[\lambda_1,\lambda_2]}(\Delta)L^2(M))\odot M[\lambda_1,\lambda_2]\subset X$ we can define the maps $q_1$ and $q_2$ to $\H_1$ and $\H_2$, respectively, by
    \begin{equation*}
        q_1(\eta\otimes x)=\delta_1(\eta)x\text{ and } q_2(\eta\otimes x)=\delta_2(\eta)x.
    \end{equation*}
    Because $\delta_1$ and $\delta_2$ are bounded and right multiplication is bounded in the operator norm, we can boundedly extend $q_1$ and $q_2$ to $X$. 

    Using Lemma \ref{lem:uniqueness_algebraic_calculation} we can now show that for all $x,y\in X$ we have
    \begin{align} \label{eq:uniqueness_initial_inpr_equality}
        \ip{q_1(T(x)),q_1(T^{-1}(y))}_{\H_1}&+\ip{q_1(T^{-1}(x)),q_1(T(y))}_{\H_1}\\
        =&\ip{q_2(T(x)),q_2(T^{-1}(y))}_{\H_2}+\ip{q_2(T^{-1}(x)),q_2(T(y))}_{\H_2}.\nonumber
    \end{align}
    Indeed, for all $j\in \{1,2\}$, $\eta,\xi\in (\1_{[\lambda_1,\lambda_2]}(\Delta)L^2(M))$ and $a,b\in M[\lambda_1,\lambda_2]$ we have
    \begin{align*}
        &\ip{q_j(T(\eta\otimes a)),q_j(T^{-1}(\xi\otimes b))}_{\H_j}+\ip{q_j(T^{-1}(\eta\otimes a)),q_j(T(\xi\otimes b))}_{\H_j}\\
        =&\ip{\delta_j(\Delta^{1/4}\eta)\sigma^{\phi}_{-i/4}(a),\delta_j(\Delta^{-1/4}\xi)\sigma^{\phi}_{i/4}(b)}_{\H_j}+\ip{\delta_j(\Delta^{-1/4}\eta)\sigma^{\phi}_{i/4}(a),\delta_j(\Delta^{1/4}\xi)\sigma^{\phi}_{-i/4}(b)}_{\H_j}\\
        =&\ip{\delta_j(\Delta^{1/4}\eta)\sigma^{\phi}_{-i/4}(ab^\ast),\delta_j(\Delta^{-1/4}\xi)}_{\H_j}+\ip{\delta_j(\Delta^{-1/4}\eta)\sigma^{\phi}_{i/4}(ab^\ast),\delta_j(\Delta^{1/4}\xi))}_{\H_j}\\
        =&\ip{\delta_j(\Delta^{-1/4}(\eta \phi^{1/2}ab^\ast)),\delta_j(\Delta^{1/4}\xi)}_{\H_j}+\ip{\delta_j(\Delta^{1/4}\eta),\delta_j(\Delta^{-1/4}\xi J(\Delta^{-1/4}(\phi^{1/2}ab^\ast)))}_{\H_j}\\
        &-\ip{\delta_j(J(\Delta^{1/4}\xi)\Delta^{1/4}\eta),\delta_j(J(\Delta^{-1/4}\phi^{1/2}ab^\ast))}_{\H_j},
    \end{align*}
    where the last step follows from Lemma \ref{lem:uniqueness_algebraic_calculation} and the fact that $va=v\cdot J\pi_r(\phi^{1/2}a)^\ast J$ for $v\in \H_j$ and $a\in M$. Since we have for all $\eta,\xi \in L^2(M)$ that 
    \begin{equation*}
        \ip{\delta_1(\eta),\delta_1(\xi)}_{\H_1}=\ip{\eta,\delta_1^\ast\delta_1(\xi)}=\ip{\eta,\delta_2^\ast\delta_2(\xi)}=\ip{\delta_2(\eta),\delta_2(\xi)}_{\H_2},
    \end{equation*}
    we can now conclude that
    \begin{align*}
        \ip{q_1(T(\eta\otimes a))&,q_1(T^{-1}(\xi\otimes b))}_{\H_1}+\ip{q_1(T^{-1}(\eta\otimes a)),q_1(T(\xi\otimes b))}_{\H_1}\\
        &=\ip{q_2(T(\eta\otimes a)),q_2(T^{-1}(\xi\otimes b))}_{\H_2}+\ip{q_2(T^{-1}(\eta\otimes a)),q_2(T(\xi\otimes b))}_{\H_2}.
    \end{align*}
    By linearity and density of $(\1_{[\lambda_1,\lambda_2]}(\Delta)L^2(M))\odot M[\lambda_1,\lambda_2]$ in $X$ we find that \eqref{eq:uniqueness_initial_inpr_equality} holds. 

    The next part of the proof is to show that \eqref{eq:uniqueness_initial_inpr_equality} implies that 
    \begin{equation*}
        \ip{q_1(x),q_1(y)}_{\H_1}=\ip{q_2(x),q_2(y)}_{\H_2}
    \end{equation*} for $x,y\in X$. For this, we consider the operator $Te^{-sT^2}$ for $s>0$, defined by holomorphic functional calculus. We start with the observation that
    \begin{align}
    \begin{split}\label{eq:uniqueness_integrand}
        &\ip{q_j(T(Te^{-sT^2}(x))),q_j(T^{-1}(Te^{-sT^2}(y)))}_{\H_j}+\ip{q_j(T^{-1}(Te^{-sT^2}(x))),q_j(T(Te^{-sT^2}(y)))}_{\H_j}\\
        &\quad=-\frac{d}{ds}\ip{q_j(e^{-sT^2}(x)),q_j(e^{-sT^2}(y))}_{\H_j}
        \end{split}
    \end{align}
    for $j\in \{1,2\}$ and $x,y\in X$. Since $T$ is bounded and $\spec(T)\subset(0,\infty)$, we know that $\lim_{s\to \infty}\norm{e^{-sT^2}}=0$ \cite[Theorem 6.24]{Nee22}. Consequently, we have for $j\in \{1,2\}$ and $x,y\in X$ that
    \begin{equation*}
        -\lim_{r\to\infty}\int_0^r \frac{d}{ds}\ip{q_j(e^{-sT^2}(x)),q_j(e^{-sT^2}(y))}_{\H_j}\,ds=\ip{q_j(x),q_j(y)}_{\H_j}.
    \end{equation*}
    Since the integrand of the above integral is equal for $j=1$ and $j=2$ by \eqref{eq:uniqueness_initial_inpr_equality} and \eqref{eq:uniqueness_integrand}, we deduce that 
    \begin{equation*}
        \ip{q_1(x),q_1(y)}_{\H_1}=\ip{q_2(x),q_2(y)}_{\H_2}
    \end{equation*}
    for all $x,y\in X$. Therefore, we have for all $\eta,\xi\in (\1_{[\lambda_1,\lambda_2]}(\Delta)L^2(M))$ and $a,b\in M[\lambda_1,\lambda_2]$ that
    \begin{equation*}
        \ip{\delta_1(\eta)a,\delta_1(\xi)b}_{\H_1}=\ip{\delta_2(\eta)a,\delta_2(\xi)b}_{\H_2}.
    \end{equation*}

    So far we have shown that $\Theta$ preserves the inner product on certain subsets of $\H_1$, and the goal is to extend this to all of $\H_1$. This takes a few steps. First, note that $\Theta$ preserves the inner product on all of 
    \begin{equation*}
        \mathrm{lin}\{\delta_1(\eta)a|\lambda_1,\lambda_2>0, \eta\in (\1_{[\lambda_1,\lambda_2]}(\Delta)L^2(M)), a\in M[\lambda_1,\lambda_2]\},
    \end{equation*}
    since for each $\eta_1,\eta_2\in \bigcup_{\lambda_1,\lambda_2>0}(\1_{[\lambda_1,\lambda_2]}(\Delta)L^2(M))$ and $a_1,a_2\in \bigcup_{\lambda_1,\lambda_2>0}M[\lambda_1,\lambda_2]$ we can find $\lambda_1',\lambda_2'>0$ such that $\eta_1,\eta_2\in(\1_{[\lambda_1',\lambda_2']}(\Delta)L^2(M))$ and $a_1,a_2\in M[\lambda_1',\lambda_2']$. Next, since $\bigcup_{\lambda_1,\lambda_2>0}(\1_{[\lambda_1,\lambda_2]}(\Delta)L^2(M))$ is dense in $L^2(M)$ and $\delta_1$ is bounded, we can extend this to $\mathrm{lin}\{\delta_1(\eta)a|\eta\in L^2(M),a\in \bigcup_{\lambda_1,\lambda_2>0}M[\lambda_1,\lambda_2]\}$ and subsequently to $\mathrm{lin}\{\delta_1(\eta)a|\eta\in L^2(M),a\in M\}$ because $\bigcup_{\lambda_1,\lambda_2>0}M[\lambda_1,\lambda_2]$ is weak* dense in $M$. Lastly, by property (c) of a first-order differential calculus we conclude that $\Theta$ is isometric on all of $\H_1$. 

    We will finish the proof by discussing the other desired properties of $\Theta$. By property (c) of a first-order differential calculus, $\mathrm{lin}\{\delta_2(\eta)a|\eta\in L^2(M), a\in M\}$ is dense in $\H_2$. Because the image of an isometric map is closed, we know that $\Theta$ is surjective and therefore that it is a linear isometric isomorphism. By property (b) of a first-order differential calculus it is a unitary bimodule map, and it is clear that it intertwines $\J_1$ and $\J_2$.
\end{proof}

\section{Derivations for Quantum Markov Semigroups with Unbounded Generators}\label{sec:unbounded}

In this section we study derivations for quantum Markov semigroups that are not necessarily uniformly continuous. In this case the generator can be unbounded, and it is convenient to work with the associated quadratic forms on $L^2(M)$, which we call quantum Dirichlet forms. We show that the bounded vectors in the form domain form an algebra (\Cref{thm:algebra_Dirichlet_form}), which gives a suitable domain for a derivation. We then show that there exists a (possibly unbounded) first-order differential calculus associated with our given quantum Dirichlet form (\Cref{thm:existence_deriv_unbounded}).

We keep the notation from the previous section. In particular, $M$ is a von Neumann algebra and $\phi$ is a normal faithful state on $M$.

Let us recall some basic definitions concerning quadratic forms on Hilbert spaces. A quadratic form on a Hilbert space $H$ is a map $q\colon H\to [0,\infty]$ such that
\begin{itemize}
    \item $q(\lambda \xi)=\abs{\lambda}^2 q(\xi)$
    \item $q(\xi+\eta)+q(\xi-\eta)=2q(\xi)+2q(\eta)$
\end{itemize}
for all $\xi,\eta\in H$ and $\lambda\in \IC$. The form $q$ is called closed if it is lower semicontinuous and densely defined if $\dom(q)=\{\xi\in H\mid q(\xi)<\infty\}$ is dense.

If $q$ is a quadratic form, it gives rise to a bilinear form on $\dom(q)\times \dom(q)$ by the polarization identity. Vice versa, the diagonal of a bilinear form extended by $\infty$ to the complement of its domain is a quadratic form. We will use both viewpoints interchangeably and denote both objects by the same symbol.

The generator of a closed densely defined quadratic form $q$ is the positive self-adjoint operator $L$ given by
\begin{align*}
    \dom(L)&=\{\xi\in \dom(q)\mid \exists\eta\in H\,\forall \zeta\in \dom(q)\colon q(\xi,\zeta)=\langle \eta,\zeta\rangle\},\\
    L\xi&=\eta.
\end{align*}
Conversely, if $L$ is a positive self-adjoint operator, then
\begin{equation*}
    q\colon H\to [0,\infty],\,q(\xi)=\begin{cases}\norm{L^{1/2}\xi}^2&\text{if }\xi\in \dom(L^{1/2}),\\\infty&\text{otherwise}\end{cases}
\end{equation*}
is a closed densely defined quadratic form with generator $q$. 

To describe the quadratic forms associated with symmetric Markov semigroups, we need the following piece of notation. For $a\in L^2(M)$ let $a\wedge \phi^{1/2}$ be the projection onto the closed convex cone $\phi^{1/2}-L^2_+(M)$.

\begin{definition}
A closed densely defined quadratic form $\E\colon L^2(M)\to [0,\infty]$ is called \emph{conservative Dirichlet form} if 
\begin{itemize}
    \item $\E(Ja)=\E(a)$ for all $a\in L^2(M)$,
    \item $\E(a\wedge \phi^{1/2})\leq \E(a)$ for all $a\in L^2(M)$,
    \item $\E(\phi^{1/2})=0$.
\end{itemize}
The form $\E$ is called conservative completely Dirichlet form or \emph{quantum Dirichlet form} if for every $n\in\mathbb N$ the quadratic form
\begin{equation*}
    \E^{(n)}\colon L^2(M_n(M))\to [0,\infty],\,\E^{(n)}([a_{jk}])=\sum_{j,k=1}^n \E(a_{jk})
\end{equation*}
is a conservative Dirichlet form.
\end{definition}

There is a one-to-one correspondence between quantum Dirichlet forms on $L^2(M)$ and KMS-symmetric quantum Markov semigroups on $M$ (see \cite[Theorem 4.11]{Cip97}, \cite[Theorem 5.7]{GL95}): If $(\Phi_t)$ is a KMS-symmetric quantum Markov semigroup on $M$ and $\L_2$ the KMS implementation of its generator on $L^2(M)$, then the quadratic form associated with $\L_2$ is a quantum Dirichlet form. Vice versa, every quantum Dirichlet form arises this way.

\begin{theorem}\label{thm:algebra_Dirichlet_form}
Let $M$ be a von Neumann algebra, $\phi$ a normal faithful state on $M$ and $\E$ a quantum Dirichlet form on $L^2(M)$. If $a\in \dom(\E)\cap M\phi^{1/2}$, $b\in \dom(\E)\cap \phi^{1/2}M$, then $ab\in \dom(\E)$ and
\begin{equation*}
    \E(ab)^{1/2}\leq \norm{\pi_l(a)}\E(b)^{1/2}+\E(a)^{1/2}\norm{\pi_r(a)}.
\end{equation*}
In particular, $\dom(\E)\cap M\phi^{1/2}\cap \phi^{1/2}M$ with the involution $J$ is a $\ast$-algebra.
\end{theorem}
\begin{proof}
Let $(T_t)$ be the strongly continuous semigroup associated with $\E$ and let $\E_t(a)=\frac 1 t\langle a,a-T_t a\rangle$ for $a\in L^2(M)$. By the spectral theorem, $\E_t(a)\nearrow \E(a)$ as $t\searrow 0$.

Let $(\Phi_t)$ be the KMS-symmetric quantum Markov semigroup on $M$ associated with $(T_t)$. Since $\Phi_t$ is completely positive, $\frac 1t(I-\Phi_t)$ is conditionally completely negative. Thus $\frac 1 t(I-\Phi_t)$ generates a quantum Markov semigroup on $M$, which is clearly KMS-symmetric, and the associated symmetric Markov semigroup on $L^2(M)$ has generator $\frac 1 t(I-T_t)$.

By \Cref{thm:existence:deriv_bounded} there exists a bounded first-order differential calculus $(\H_t,\J_t,\delta_t)$ such that $\E_t(a)=\norm{\delta_t(a)}_{\H_t}^2$ for $a\in L^2(M)$. Thus, if $a\in \dom(\E)\cap M\phi^{1/2}$ and $b\in \dom(\E)\cap \phi^{1/2}M$, then
\begin{align*}
    \E_t(ab)^{1/2}&=\norm{\delta_t(ab)}_{\H_t}\\
    &=\norm{\pi_l(a)\delta_t(b)+\delta_t(a)J\pi_r(b)^\ast J}_{\H_t}\\
    &=\norm{\pi_l(a)}\norm{\delta_t(b)}_{\H_t}+\norm{\delta(a)}_{\H_t}\norm{\pi_r(b)}\\
    &=\norm{\pi_l(a)}\E_t(b)^{1/2}+\E(a)^{1/2}\norm{\pi_r(b)}.
\end{align*}
The claim follows by taking the limit $t\searrow 0$ of both sides.
\end{proof}

\begin{remark}
In contrast to the case of GNS-symmetric quantum Markov semigroups \cite[Theorem 6.3]{Wir22b} we could not show that $\dom(\E)\cap M\phi^{1/2}\cap \phi^{1/2}M$ is a form core for $\E$. The space $\dom(\E)\cap \phi^{1/4}M\phi^{1/4}$ is always a form core, but we do not expect it to be an algebra in general.
\end{remark}

\begin{theorem}\label{thm:existence_deriv_unbounded}
Let $M$ be a von Neumann algebra and $\phi$ a faithful normal state on $M$. If $\E$ is a quantum Dirichlet form on $L^2(M)$, then there exists a Hilbert space $\H$ with commuting left and right actions of $M$, an anti-unitary involution $\J\colon \H\to\H$ such that
\begin{equation*}
    \J(x\xi y)=y^\ast(\J\xi)x^\ast
\end{equation*}
for $x,y\in M$ and $\xi\in \H$, a closed operator $\delta\colon \dom(\E)\to \H$ such that $\J\delta=\delta J$ and
\begin{equation*}
    \delta(ab)=\pi_l(a)\cdot\delta(b)+\delta(a)\cdot J\pi_r(a)^\ast J
\end{equation*}
for $a\in \dom(\E)\cap M\phi^{1/2}$, $b\in \dom(\E)\cap \phi^{1/2}M$, and
\begin{equation*}
    \E(a,b)=\langle\delta(a),\delta(b)\rangle_\H    
\end{equation*}
for $a,b\in \dom(\E)$.
\end{theorem}

\begin{remark}
By commuting left and right actions of $M$ on $\H$ we mean unital $\ast$-homomorphisms $\pi_l^\H\colon M\to B(\H)$, $\pi_r^\H\colon M^\op\to B(\H)$ with commuting images. As usual, we write $x\xi y$ for $\pi_l^\H(x)\pi_r^\H(y^\op)\xi$.

We do not claim that the actions of $M$ on $\H$ are normal so that $\H$ is in general not a correspondence. This is not only an artefact of our proof or a feature of KMS symmetry specifically, but happens even for symmetric Markov semigroups on commutative von Neumann algebras. For GNS-symmetric quantum Markov semigroups, the normality of the action is linked to the existence of the carré du champ (see \cite[Section 7]{Wir22b}).
\end{remark}

\begin{proof}
Let $(T_t)$ be the strongly continuous semigroup associated with $\E$. As in the proof of \Cref{thm:algebra_Dirichlet_form}, we consider the quantum Dirichlet form $\E_t$ given by $\E_t(a)=\frac 1 t \langle a,a-T_t a\rangle$ and the associated first-order differential calculus $(\H_t,\J_t,\delta_t)$. We will use an ultraproduct construction to define $(\H,\J,\delta)$ for $\E$.

Choose $\omega\in\beta\IN\setminus \IN$ and a null sequence $(t_n)$ in $(0,\infty)$ and let $\H$ be the ultraproduct $\prod_{n\to\omega}\H_{t_n}$. We write $[\xi_n]$ for the equivalence class of $(\xi_n)$ in $\H$.

We can define commuting left and right actions of $M$ on $\H$ by
\begin{align*}
    x[\xi_n]y=[x\xi_n y].
\end{align*}
As noted in the previous remark, these actions are not necessarily normal.

Moreover, let 
\begin{equation*}
    \J\colon \H\to \H,\,\J[\xi_n]=[\J_{t_n}\xi_n].
\end{equation*}
It is easy to verify that $\J$ is an anti-unitary involution such that $\J(x\xi y)=y^\ast (\J\xi)x^\ast$ for $x,y\in M$ and $\xi\in\H$.

Finally, if $a\in\dom(\E)$, let $\delta(a)=[\delta_{t_n}(a)]$. Since
\begin{equation*}
    \norm{\delta_{t_n}(a)}_{\H_{t_n}}^2=\E_{t_n}(a)\leq \E(a),
\end{equation*}
the map $\delta\colon \dom(\E)\to \H$ is well-defined. Furthermore,
\begin{equation*}
    \langle \delta(a),\delta(b)\rangle_\H=\lim_{n\to\omega}\langle \delta_{t_n}(a),\delta_{t_n}(b)\rangle_{\H_{t_n}}=\lim_{n\to\omega}\E_{t_n}(a,b)=\E(a,b)
\end{equation*}
The operator $\delta$ is closed because $\E$ is closed. The other properties of $\delta$ can be checked componentwise.
\end{proof}

\begin{remark}
At this point we do not know if the triple $(\H,\J,\delta)$ is uniquely determined. Given the known results for tracially symmetric or more generally GNS-symmetric quantum Markov semigroups, it seems reasonable to suspect that the left and right action are not uniquely determined on $M$, but only on some $\ast$-subalgebra.
\end{remark}

In the case of tracially symmetric or more generally GNS-symmetric quantum Markov semigroups it can be useful to realize the bimodule $\H$ inside $L^2(\hat M)$ for some von Neumann algebra $\hat M$ containing $M$ with (faithful normal) expectation (see \cite{JRS,Wir22b}). We will show that in certain cases such a von Neumann algebra $\hat M$ can also be found for KMS-symmetric semigroups, although the construction seems less canonical.

Clearly, a necessary condition for the existence is that $M$ acts normally on $\H$. Since we do not have a criterion for this to happen in terms of the semigroup, we formulate the result purely in terms of correspondences.

\begin{proposition}
Let $M$ be a von Neumann algebra, $\phi$ a faithful normal state on $M$ and $(\H,\J)$ a self-dual $M$-$M$-correspondence. If $M$ is semi-finite or $M$ has finite-dimensional center and $\H$ has finite index, then there exists a von Neumann algebra $\hat M$ containing $M$, a faithful normal conditional expectation $E\colon \hat M\to M$, and an isometric bimodule map $V\colon \H\to L^2(\hat M)$ such that $V\J=J_{\phi\circ E}V$, 
\end{proposition}
\begin{proof}
We will show that in either case there is a strongly continuous unitary group $(\U_t)$ on $\H$ such that $\U_t^\H(x\xi y)=\sigma_t^\phi(x) (\U_t^\H \xi)\sigma_t^\phi(y)$ for $x,y\in M$, $\xi \in \H$, $t\in\IR$ and $\J\U^\H_t=\U_t^\H\J$ for $t\in\IR$. This makes $(\H,\J,(\U_t^\H))$ into a Tomita correspondence in the terminology of \cite{Wir22b}, and the result follows from \cite[Section 4.2]{Wir22b}.

First assume that $M$ is semi-finite and let $\tau$ be a normal semi-finite faithful trace on $M$. There exists $\rho\in L^1_+(M,\tau)$ such that $\phi=\tau(\,\cdot\,\rho)$. Let $\U_t^\H\xi=\rho^{it}\xi\rho^{-it}$. Since the actions of $M$ on $\H$ are normal, $(\U_t^\H)$ is a strongly continuous unitary group. Moreover,
\begin{equation*}
    \U_t^\H(x\xi y)=\rho^{it}x\xi y\rho^{-it}=\sigma^\phi_t(x)\rho^{it}\xi\rho^{-it}\sigma^\phi_t(y)=\sigma^\phi_t(x)(\U_t^\H\xi)\sigma^\phi_t(y)
\end{equation*}
and
\begin{equation*}
    \U_t^\H \J\xi=\rho^{it}(\J\xi)\rho^{-it}=\J(\rho^{it}\xi \rho^{-it})=\J\U_t^\H\xi
\end{equation*}
for $x,y\in M$, $\xi\in \H$ and $t\in\IR$. This settles the claim for semi-finite von Neumann algebras.

Now assume that $M$ has finite-dimensional center and $\H$ has finite index. In this case we use Longo's construction from \cite{Lon18}, which we briefly recall.

Write $\pi_l^\H$ and $\pi_r^\H$ for the left and right action of $M$ on $\H$ and let $\epsilon\colon \pi_r^\H(M)^\prime\to \pi_l^\H(M)$ be the minimal conditional expectation. Define $\Delta_\H$ to be the spatial derivative of $\phi\circ (\pi_l^\H)^{-1}\circ \epsilon$ with respect to $\phi\circ (\pi_r^\H)^{-1}$, that is,
\begin{equation*}
    \Delta_\H=\frac{d(\phi\circ(\pi_l^\H)^{-1}\circ \epsilon)}{d(\phi\circ (\pi_r^\H)^{-1})}.
\end{equation*}
Here $(\pi_l^\H)^{-1}$ and $(\pi_r^\H)^{-1}$ are to be understood as follows: There exist central projections $p,q\in M$ such that $\ker \pi_l^\H=(1-p)M$, $\ker \pi_r^\H=(1-q)M$. Then $\pi_l^\H$ restricted to $pM$ (resp. $\pi_r^\H$ restricted to $qM$) is injective, and we write $(\pi_l^\H)^{-1}$ (resp. $(\pi_r^\H)^{-1})$ for the inverse of this restriction.

By definition of the spatial derivative, $\Delta_\H$ is a non-singular positive self-adjoint operator on $\H$.

Since $M$ has finite-dimensional center, there are central projections $e_1,\dots,e_n\in M$ such that $M\cap M'=\bigoplus_j \IC e_j$. Let $\H_{ij}=e_i \H e_j$. This is an $e_i M$-$e_j M$-correspondence and $\H=\bigoplus_{i,j}\H_{ij}$ as Hilbert spaces. Let $d_{ij}=\sqrt{\mathrm{Ind}(\H_{ij})}$ and 
\begin{equation*}
D_\H=\sum_{i,j}d_{ij}\pi_l^\H(e_i)\pi_r^\H(e_j^\op).
\end{equation*}
Finally let $\U_t^\H=\Delta_\H^{it}D_\H^{it}$. The unitary group $(\U_t)$ satisfies
\begin{equation*}
    \U_t(x\xi y)=\sigma^\phi_t(x)(\U_t\xi)\sigma^\phi_t(y)
\end{equation*}
for $x,y\in M$, $\xi\in \H$ and $t\in\IR$ by \cite[Theorem 2.3]{Lon18}.

Moreover, let
\begin{equation*}
    T\colon \H\to \overline{\H},\,\xi\mapsto\overline{\J\xi}.
\end{equation*}
By definition of $\J$, the map $T$ is a unitary bimodule map. It follows from \cite[Theorem 2.3]{Lon18} that
\begin{equation*}
    \overline{\J\U_t^\H \xi}=T\U_t^\H\xi=\U_t^{\bar\H}T\xi=\U_t^{\bar\H}\, \overline{\J\xi}=\overline{\U_t^\H \J\xi}
\end{equation*}
for $\xi \in \H$ and $t\in\IR$. Thus $\J$ commutes with $(\U_t^\H)$.
\end{proof}

%\begin{corollary}
%Let $M$ be a semi-finite von Neumann algebra and $\phi$ a faithful normal state on $M$. If $(\Phi_t)$ is a uniformly continuous quantum Markov semigroup on $M$ with generator $\L$, then there exists von Neumann algebra $\hat M$ containing $M$, a faithful normal conditional expectation $E\colon \hat M\to M$ and a bounded operator $\delta\colon L^2(M)\to L^2(\hat M)$ satisfying
%\begin{itemize}
%    \item $\delta\circ J_\phi=J_{\phi\circ E}\circ\delta$,
%    \item $\delta(ab)=a\delta(b)+\delta(a)b$ for all $a\in M\phi^{1/2}$, $b\in %\phi^{1/2}M$
%\end{itemize}
%such that
%\begin{equation*}
%    \L_2=\delta^\ast\delta.
%\end{equation*}
%\end{corollary}

\appendix
\section{Alternative proof of Theorem \ref{thm:commutator_sqrt_matrix}}

The proof of Theorem \ref{thm:commutator_sqrt_matrix} shows that this result can be easily obtained from Theorem \ref{thm:existence_derivation_matrix}, and therefore that one can prove it in a basis-independent fashion. On the other hand, it gives very little intuition for the $V_j$ mentioned in the theorem. To provide a bit more feeling for the $V_j$, we include an alternative proof of the result.

Recall that by \cite[Theorem 4.4]{AC21} the generator $\L$ of a KMS-symmetric quantum Markov semigroup on $M_n(\IC)$ can be written as
\begin{equation*}\label{eq:rep_KMS_generator}
    \L(A)=(1+\sigma_{-i/2})^{-1}(\Psi(I_n))A+A(1+\sigma_{i/2})^{-1}(\Psi(I_n))-\Psi(A)\tag{$\ast$}
\end{equation*}
with a KMS-symmetric completely positive map $\Phi\colon M_n(\IC)\to M_n(\IC)$.

\begin{proof}[Alternative proof of Theorem \ref{thm:commutator_sqrt_matrix}]
    Let $\Psi$ be a KMS-symmetric completely positive map such that \eqref{eq:rep_KMS_generator} holds, i.e. \begin{equation*}
        \L(A)=(1+\sigma_{-i/2})^{-1}(\Psi(I_n))A+A(1+\sigma_{i/2})^{-1}(\Psi(I_n))-\Psi(A).
    \end{equation*}
    By Lemma \ref{lem:properties_V-trafo_matrix}(ii), $\check{\Psi}$ is also completely positive and KMS-symmetric. Next, we define the completely positive map $\Xi$ by
    \begin{equation*}
        \Xi(A)=\rho^{1/4}\check{\Psi}(\rho^{-1/4}A\rho^{-1/4})\rho^{1/4}
    \end{equation*}
    for all $A\in M_n{\IC}$. Now for all $A,B\in M_n(\IC)$ we have
    \begin{equation*}
        \tr(A\Xi(B))=\tr(\rho^{-1/4}A\rho^{-1/4}\rho^{1/2}\check{\Psi}(\rho^{-1/4}B\rho^{-1/4})\rho^{1/2})=\tr(\Xi(A)B)
    \end{equation*}
    by the KMS-symmetry of $\check{\Psi}$.     
    
    Let $V_1, \dots, V_N$ be the Kraus representation of $\Xi$, meaning that
    \begin{equation*}
        \Xi(A)=\sum_{j=1}^N V_j^\ast A V_j
    \end{equation*}
    for all $A\in M_n(\IC)$. Since $\tr(A\Xi(B))=\tr(\Xi(A)B)$, we can assume without loss of generality that for each $j$ there exists an index $j^\ast$ such that $V_j^\ast=V_{j^\ast}$ and $(j^\ast)^\ast=j$. Note that $\sigma_{-i/4}(V_1),\dots,\sigma_{-i/4}(V_N)$ is a Kraus representation of $\check{\Psi}$. 
    Calculating $\sum_j\ip{[V_j,A],[V_j,B]}$ then gives
    \begin{align*}
        \sum_j\ip{[V_j,A],[V_j,B]}_\rho=&\sum_j\tr(\rho^{1/2}(A^\ast V_j^\ast - V_j^\ast A^\ast)\rho^{1/2}(V_jB-BV_j))\\
        =&\sum_j\tr(\rho^{1/2}A^\ast \rho^{1/2}(\sigma_{i/2}(V_j^\ast)V_jB+BV_j\sigma_{-i/2}(V_j)\\
        &\quad-\sigma_{i/2}(V_j^\ast)BV_j-V_jB\sigma_{-i/2}(V_j^\ast)))\\
        =&\sum_j\tr(\rho^{1/2}A^\ast \rho^{1/2}(\sigma_{i/2}(V_j^\ast)V_j B+BV_j^\ast\sigma_{-i/2}(V_j)\\
        &\quad-\sigma_{i/2}(V_j^\ast)BV_j-V_j^\ast B\sigma_{-i/2}(V_j))).
    \end{align*}
    We observe that
    \begin{align*}
        \sum_j\sigma_{i/2}(V_j^\ast)BV_j+V_j^\ast B\sigma_{-i/2}(V_j)=&\sum_j\sigma_{i/4}(\sigma_{-i/4}(V_j)^\ast)B\sigma_{i/4}(\sigma_{-i/4}(V_j))\\
        &-\sigma_{-i/4}(\sigma_{-i/4}(V_j)^\ast) B\sigma_{-i/4}(\sigma_{-i/4}(V_j))\\
        =&\W(\check{\Psi})(B)
    \end{align*}
    and consequently that
    \begin{equation*}
        (1+\sigma_{i/2})(\sum_jV_j^\ast\sigma_{-i/2}(V_j))=\sum_j\sigma_{i/2}(V_j^\ast)V_j+V_j^\ast \sigma_{-i/2}(V_j)=\W(\check{\Psi})(I_n).
    \end{equation*}
    Therefore we have
    \begin{align*}
        \sum_jV_j^\ast\sigma_{-i/2}(V_j)&=(1+\sigma_{i/2})^{-1}(\Psi(I_n)),\\
        \sum_j\sigma_{i/2}(V_j^\ast)V_j&=(1+\sigma_{-i/2})^{-1}(\Psi(I_n)).
    \end{align*}
    Coming back to our original expression we find
    \begin{align*}
        \sum\nolimits_j\ip{[V_j,A],[V_j,B]}_\rho&=\tr(\rho^{1/2}A^\ast \rho^{1/2}((1+\sigma_{-i/2})^{-1}(\Psi(I_n)) B\\
        &\quad+B(1+\sigma_{i/2})^{-1}(\Psi(I_n))-\Psi(B)))\\
        &=\langle \L(A),B\rangle_\rho.\qedhere
    \end{align*}
\end{proof}

\DeclareFieldFormat[article]{citetitle}{#1}
\DeclareFieldFormat[article]{title}{#1}

\printbibliography

%\bibliography{references}
%\bibliographystyle{alpha}

\end{document}